\documentclass[11pt, a4paper]{amsart}
\usepackage{amssymb}
\usepackage{graphicx}
\newtheorem{lemm}{Lemma}[section]
\newtheorem{coro}{Corollary}[section]
\newtheorem{prop}{Proposition}[section]
\newtheorem{theo}{Theorem}[section]
\theoremstyle{definition}
\newtheorem{defi}{Definition}[section]
\newtheorem{rema}{Remark}[section]

\numberwithin{equation}{section}

\def\Bbb{\mathbb}
\def\fb{\partial\{u>0\}}
\def\uep{u^\varepsilon}
\def\ep{\varepsilon}

\def\a{\alpha}

\def\fint{\operatorname {--\!\!\!\!\!\int\!\!\!\!\!--}}

\def\uepj{u^{\varepsilon_j}}
\def\pepj{P_{\varepsilon_j}}

\def\rn1{\Bbb R^{N+1}}
\def\rn{\Bbb R^{N}}

\def\epj{\varepsilon_j}

\def\fepj{f^{{\varepsilon}_j}}
\def\fep{f^{{\varepsilon}}}

\def\a*{\alpha^*_\ep}

\def\R{\mathbb R}

\def\lstar{\lambda^*}
\def\lone{\lambda_{\min}}
\def\ltwo{\lambda_{\max}}
\def\H{{\mathcal{H}}^{N-1}}
\def\RR{{\mathbb {R}}}
\def\pint{\operatorname {--\!\!\!\!\!\int\!\!\!\!\!--}}

\begin{document}
\title[regularity of the interface in a free boundary
problem for the $p(x)$-Laplacian]{Weak solutions and regularity of the interface in an inhomogeneous free boundary
problem for the $p(x)$-Laplacian}
\author[Claudia Lederman]{Claudia Lederman}
\author[Noemi Wolanski]{Noemi  Wolanski}
\address{IMAS - CONICET and Departamento  de
Ma\-te\-m\'a\-ti\-ca, Facultad de Ciencias Exactas y Naturales,
Universidad de Buenos Aires, (1428) Buenos Aires, Argentina.}
\email[Claudia Lederman]{clederma@dm.uba.ar} \email[Noemi
Wolanski]{wolanski@dm.uba.ar}
\thanks{Supported by the Argentine Council of Research CONICET under the project PIP625, Res. 960/12,  UBACYT 20020100100496 and
ANPCyT PICT 2012-0153.}

\keywords{Free boundary problem, variable exponent spaces,
regularity of the free boundary, singular perturbation, inhomogeneous problem.
\\
\indent 2010 {\it Mathematics Subject Classification.} 35R35,
35B65, 35J60, 35J70}\maketitle
%35R35 Free boundary problems for PDE
%35B65 Smoothness and regularity of solutions of PDE
%35J60 Nonlinear elliptic equations
%35J70 Degenerate elliptic equations

%%%%%%%%%% abstract %%%%%%%%%%%%%%%%%%%%%%%%%%%%%%%%%%%%%
\begin{abstract}
In this paper we study a one phase  free boundary
problem for the $p(x)$-Laplacian with non-zero right hand side.
We  prove that the free boundary  of a weak
solution is a $C^{1,\alpha}$ surface in a neighborhood of every
``flat'' free boundary point.
We also obtain further regularity results on the free boundary,
under further regularity assumptions on the data.
We apply these results to limit functions of an inhomogeneous
singular perturbation problem for the $p(x)$-Laplacian that we
studied in \cite{LW4}.
\end{abstract}

\bigskip

%%%%%%%%%%%%%% end of abstract %%%%%%%%%%%%%%%%%%%%%%%%%%%%%%%%%%%%%%%%%

%%%%%%%%%  introduction %%%%%%%%%%%%%%%%%%%%%%%%%%%%%%%%%%%%%%%%%%%%%%
\begin{section}{Introduction}
\label{sect-intro}

In this paper we study the following inhomogeneous free boundary
problem for the $p(x)$-Laplacian: $u\ge 0$ and
\begin{equation}\label{fbp-px}
\tag{$P(f,p,{\lambda}^*)$}
\begin{cases}
\Delta_{p(x)}u:=\mbox{div}(|\nabla u(x)|^{p(x)-2}\nabla
u)= f & \mbox{in }\{u>0\}\\
u=0,\ |\nabla u| = \lambda^*(x) & \mbox{on }\partial\{u>0\}.
\end{cases}
\end{equation}

The $p(x)$-Laplacian serves as a model for a stationary
non-newtonian fluid with properties depending on the point in the
region where it moves. For example, such a situation corresponds
to an electrorheological fluid. These are fluids such that their
properties depend on the magnitude of the electric field applied to
it. In some cases, fluid and Maxwell's equations become uncoupled
and a single equation for the $p(x)$-Laplacian appears (see
\cite{R}).

The free boundary problem \ref{fbp-px} appears, for instance, in
the limit of a singular perturbation problem that may model
 high activation energy deflagration flames in a fluid with electromagnetic sensitivity (see \cite{LW4}). When $p(x)\equiv2$
 (in which case the $p(x)$-Laplacian coincides with the Laplacian) this singular perturbation problem was introduced by Zeldovich and Frank-Kamenetski
 in order to model these kind of flames in \cite{ZFK}. In this latter case, the right hand side $f$ may come from nonlocal effects as well as from
 external sources (see \cite{LW2}).

 The free boundary problem considered in this paper also appears in an inhomogeneous minimization problem that we study in \cite{LW5} where we prove
 that minimizers are weak solutions to \ref{fbp-px}.

 \medskip

 In the present article we prove that the free boundary $\partial\{u>0\}$ ---with $u$ a weak solution of \ref{fbp-px}--- is a smooth hypersurface in a
 neighborhood of every ``flat'' free boundary point. 

The notion of weak solution used in this paper is such that it
also includes the limits of the singular perturbation problem
described above, that we studied in \cite{LW4}, under suitable nondegeneracy conditions.
 \medskip

More precisely, in the present work we prove that the free boundary of a \emph{weak} solution to
$P(f,p,{\lambda}^*)$ (see Definition \ref{weak2}) is a $C^{1,\alpha}$ surface
near flat free boundary points (Theorems \ref{reg-weak-conflat+dens}, \ref{reg-weak-conflat+flat} and \ref{reg-weak-conflat}). As a consequence we get that  the
free boundary is  $C^{1,\alpha}$   in a neighborhood of
every point in the reduced free boundary
(Theorem \ref{reg-weak-final}).
We also obtain further regularity results on the free boundary,
under further regularity assumptions on the data (Corollary \ref{higher-reg}).

\medskip

In the particular situation of the minimization problem mentioned
above,   we prove in \cite{LW5} that the set of singular free
boundary points has   null ${\mathcal H}^{N-1}$-measure.

\medskip

The basic ideas we follow in this paper to prove the regularity of the free boundary of a weak solution were introduced by Alt and
Caffarelli in the seminal paper \cite{AC}, where the case of \emph
{distributional} weak solutions of \ref{fbp-px} with $p(x)\equiv2$
and $f\equiv0$ was studied. The treatment of a quasilinear
equation was first done in \cite{ACF} for the uniformly elliptic
case. Then, the $p$-Laplacian ($p(x)\equiv p$) was treated in
\cite{DP1}. The main difference being that a control of $|\nabla
u|$ from below close to the free boundary is needed in order to be
able to work with linear equations with the ideas of \cite{ACF}.
Both \cite{ACF} and \cite{DP1} deal with minimizers that are weak
solutions in the stronger sense of \cite{AC}. A notion of weak
solution similar to the one in the present paper was first
considered in \cite{MW1}. The case of a variable power $p(x)$ was
considered in \cite{FBMW} still for minimizers and in the
homogeneous case $f\equiv0$. The linear inhomogeneous case was
treated in \cite{GS} and \cite{Le}  for minimizers.

\medskip

 We point out that the regularity of the free boundary for the inhomogeneous problem $f\not\equiv0$  had not been obtained even in the case of $p(x)\equiv p$.

\medskip

For other references related to the free boundary problem under consideration in this paper we would like to refer the reader to \cite{C1}, \cite{C2}, \cite{CJK},  \cite{DP2}, \cite{D}, \cite{DFS}, \cite{LN}, \cite{Ly}, \cite{MoWa}, \cite{O}, \cite{OY}, \cite{Wa}, \cite{We} and the references therein. This list is by no means exhaustive.

\medskip

An outline of the paper is as follows: in Section~2 we define the
notion of weak solution to the free boundary problem
\ref{bernoulli-px} and we derive some properties of weak
solutions. In Section~3 we study the behavior of weak solutions to the free boundary problem \ref{bernoulli-px} near ``flat" free boundary points.
In Section~4 we study the regularity of the free
boundary for weak solutions to the free boundary problem
\ref{bernoulli-px}. In Section~5 we present an application of
these results to limit functions of the singular perturbation
problem  that  we studied in \cite{LW4}. Our results apply to limit functions satisfying suitable conditions that are
fulfilled, for instance, under the situation we considered in \cite{LW5}.

%%%%%%%%% preliminaries and notation %%%%%%%%%%%%%%%%%%%%%%%%%%%%%%%%%%%%%
\begin{subsection}{Preliminaries on Lebesgue and Sobolev spaces with variable
exponent}

Let $p :\Omega \to  [1,\infty)$ be a measurable bounded function,
called a variable exponent on $\Omega$ and denote $p_{\max} = {\rm
ess sup} \,p(x)$ and $p_{\min} = {\rm ess inf} \,p(x)$. We define
the variable exponent Lebesgue space $L^{p(\cdot)}(\Omega)$ to
consist of all measurable functions $u :\Omega \to \R$ for which
the modular $\varrho_{p(\cdot)}(u) = \int_{\Omega} |u(x)|^{p(x)}\,
dx$ is finite. We define the Luxemburg norm on this space by
$$
\|u\|_{L^{p(\cdot)}(\Omega)} = \|u\|_{p(\cdot)}  = \inf\{\lambda >
0: \varrho_{p(\cdot)}(u/\lambda)\leq 1 \}.
$$

This norm makes $L^{p(\cdot)}(\Omega)$ a Banach space.

There holds the following relation between $\varrho_{p(\cdot)}(u)$
and $\|u\|_{L^{p(\cdot)}}$:
\begin{align*}
\min\Big\{\Big(\int_{\Omega} |u|^{p(x)}\, dx\Big)
^{1/{p_{\min}}},& \Big(\int_{\Omega} |u|^{p(x)}\, dx\Big)
^{1/{p_{\max}}}\Big\}\le\|u\|_{L^{p(\cdot)}(\Omega)}\\
 &\leq  \max\Big\{\Big(\int_{\Omega} |u|^{p(x)}\, dx\Big)
^{1/{p_{\min}}}, \Big(\int_{\Omega} |u|^{p(x)}\, dx\Big)
^{1/{p_{\max}}}\Big\}.
\end{align*}

Moreover, the dual of $L^{p(\cdot)}(\Omega)$ is
$L^{p'(\cdot)}(\Omega)$ with $\frac{1}{p(x)}+\frac{1}{p'(x)}=1$.

Let $W^{1,p(\cdot)}(\Omega)$ denote the space of measurable
functions $u$ such that $u$ and the distributional derivative
$\nabla u$ are in $L^{p(\cdot)}(\Omega)$. The norm

$$
\|u\|_{1,p(\cdot)}:= \|u\|_{p(\cdot)} + \| |\nabla u|
\|_{p(\cdot)}
$$
makes $W^{1,p(\cdot)}$ a Banach space.

The space $W_0^{1,p(\cdot)}(\Omega)$ is defined as the closure of
the $C_0^{\infty}(\Omega)$ in $W^{1,p(\cdot)}(\Omega)$.

For  more about these spaces, see \cite{DHHR,KR} and the
references therein.
\end{subsection}

\bigskip

\begin{subsection}{Preliminaries on solutions to $p(x)$-Laplacian.} Let
$p(x)$ be as above and let $g\in L^{\infty}(\Omega)$. We say that
$u$ is a solution to
$$
\Delta_{p(x)}u = g(x) \ \mbox{ in } \ \Omega
$$
if $u\in W^{1,p(\cdot)}(\Omega)$ and,  for every  $\varphi \in
W_0^{1,p(\cdot)}(\Omega)$, there holds that
$$
\int_{\Omega} |\nabla u(x)|^{p(x)-2}\nabla u \cdot \nabla
\varphi\, dx =-\int_{\Omega} \varphi\, g(x)\, dx.
$$
Under the assumptions of the present paper (see \ref{assump}
below) it follows from \cite{Wo} that $u\in L_{\rm
loc}^{\infty}(\Omega)$.

\bigskip

For any $x\in\Omega$, $\xi,\eta\in\R^N$ fixed we have the
following inequalities
\begin{align*}
&|\eta-\xi|^{p(x)}\leq C (|\eta|^{p(x)-2} \eta-|\xi|^{p(x)-2} \xi)
(\eta-\xi)&\quad  \mbox{ if } p(x)\geq 2,\\
&
 |\eta-\xi|^2\Big(|\eta|+|\xi|\Big)^{p(x)-2}
\leq C (|\eta|^{p(x)-2} \eta-|\xi|^{p(x)-2} \xi) (\eta-\xi)&\quad
\mbox{ if } p(x)< 2.
\end{align*}
These inequalities imply that the function
$A(x,\xi)=|\xi|^{p(x)-2}\xi$ is strictly monotone. Then, the
comparison principle for the $p(x)$-Laplacian holds since it
follows from the monotonicity of $A(x,\xi)$.

\end{subsection}

\bigskip

\begin{subsection}{Assumptions}\label{assump}

\medskip

Throughout the paper we let $\Omega\subset\R^N$ be a domain.

\bigskip

\noindent{\bf Assumptions on $p(x)$.} We  assume that the
function $p(x)$ verifies
\begin{equation}\label{pminmax}
1<p_{\min}\le p(x)\le p_{\max}<\infty,\qquad x\in\Omega.
\end{equation}

Unless otherwise stated, we assume that $p(x)$ is
Lipschitz continuous in $\Omega$. In some results we assume further that $p\in W^{1,\infty}(\Omega)\cap W^{2,q}(\Omega)$.

\bigskip

\noindent{\bf Assumptions on $\lstar(x)$.} We assume that the
function $\lstar$ is continuous in $\Omega$ and verifies
\begin{equation}\label{lsminmax}
0<\lone\le\lstar(x)\le\ltwo<\infty,\qquad x\in\Omega.
\end{equation}

In our main results $\lstar(x)$ is H\"older continuous in $\Omega$.

\bigskip

\noindent{\bf Assumptions on $f(x)$.} We  assume that 
$f\in L^{\infty}(\Omega)$. In some results we assume further that $f\in W^{1,q}(\Omega)$.

\end{subsection}

\bigskip

\begin{subsection}{Notation}\ \ \newline

$\bullet$ $N$ \quad spatial dimension

$\bullet$   $\Omega\cap\partial\{ u>0 \}$ \quad free boundary

$\bullet$  $|S|$ \quad  $N$-dimensional Lebesgue measure of the
set $S$

$\bullet$ ${\mathcal H}^{N-1}$ \quad  $(N-1)$-dimensional
Hausdorff measure

$\bullet$  $B_r(x_0)$ \quad  open ball of radius $r$ and center
$x_0$

$\bullet$  $B_r$ \quad  open ball of radius $r$ and center
$0$

$\bullet$  $B_r^+=B_r\cap\{x_N>0\}, \quad B_r^-=B_r\cap\{x_N<0\}$

$\bullet$  $B'_r(x_0)$ \quad  open ball of radius $r$ and center
$x_0$ in $\R^{N-1}$

$\bullet$  $B'_r$ \quad  open ball of radius $r$ and center
$0$ in $\R^{N-1}$

$\bullet$  $\fint_{B_r(x_0)}u= {1\over {|B_r(x_0)|}}
\int_{B_r(x_0)}u\,dx$

$\bullet$  $\fint_{\partial B_r(x_0)}u= {1\over {{\mathcal
H}^{N-1} (\partial B_r(x_0))}} \int_{\partial
B_r(x_0)}u\,d{\mathcal H}^{N-1}$

$\bullet$ $\chi_{{}_S}$ \quad  characteristic function of the set
$S$

$\bullet$ $u^{+}=\text{\rm max}(u,0)$,\quad $u^{-}=\text{\rm
max}(-u,0)$

$\bullet$ $\langle\,\xi\, ,\,\eta\,\rangle$ \, and \, $\xi \cdot \eta$ \quad both denote scalar product in $\Bbb R^{N}$

%%%%%%%%%% end of preliminaries and notation %%%%%%%%%%%%%%%%%%%%%%%%%%%%%%%%
\end{subsection}

\end{section}
%%%%%%%%%%  end of introduction %%%%%%%%%%%%%%%%%%%%%%%%%%%%%%%%%%%%%

%%%%%%%%%%%%%section weak solutions to the free boundary problem%%%%%%%%%%%%%%%%%%%

\begin{section}{Weak solutions to the free boundary problem $P(f,p,{\lambda}^*)$}
\label{sect-weak-solut} In this section we define the notion of
weak solution to the free boundary problem \ref{bernoulli-px}.

We also derive some properties of the weak solutions to problem
\ref{bernoulli-px}, which will be used in the next sections, where
a theory for the regularity of the free boundary for weak
solutions will be developed.

In all the results of this section $p(x)$ will be a Lipschitz continuous function.

We first need

\begin{defi}
\label{clw2-Definition 3.1} Let $u$ be a continuous and nonnegative function in a domain $\Omega\subset \Bbb R^{N}$.
We say that $\nu$ is the exterior unit
normal to the free boundary $\Omega\cap\fb$ at a point $x_0\in\Omega\cap\fb$ in the
 measure theoretic sense, if $\nu
\in\Bbb R^N$, $|\nu|=1$
and
\begin{equation}
\label{(3.3)} \lim_{r\to 0} \frac1{r^{N}} \int_{B_r(x_0)}
|\chi_{\{u>0\}}- \chi_{\{x\,/\, \langle x-x_0,\nu\rangle
<0\}}|\,dx = 0.
\end{equation}
\end{defi}

Then we have

\begin{defi}\label{weak2} Let $\Omega\subset \Bbb R^{N}$ be a domain. Let $p$ be a
measurable function in $\Omega$ with $1<p_{\min}\le p(x)\le
p_{\max}<\infty$, $\lstar$ continuous in $\Omega$ with
$0<\lone\le\lstar(x)\le\ltwo<\infty$ and $f\in L^\infty(\Omega)$.
We call $u$ a weak solution of \ref{bernoulli-px} in $\Omega$ if
\begin{enumerate}
\item $u$ is continuous and nonnegative in $\Omega$, $u\in W^{1,p(\cdot)}(\Omega)$ and
$\Delta_{p(x)}u=f$ in $\Omega\cap\{u>0\}$.
\item For
$D\subset\subset \Omega$ there are constants $c_{\min}=c_{\min}(D)$, $C_{\max}=C_{\max}(D)$, $r_0=r_0(D)$, $0< c_{\min}\leq
C_{\max}$, $r_0>0$, such that for balls $B_r(x)\subset D$ with $x\in
\partial \{u>0\}$ and $0<r\le r_0$
$$
c_{\min}\leq \frac{1}{r}\sup_{B_r(x)} u \leq C_{\max}.
$$
\item For $\mathcal{H}^{N-1}$ a.e.
$x_0\in\partial_{\rm{red}}\{u>0\}$ (this is, for ${\mathcal
H}^{N-1}$-almost every point $x_0\in\Omega\cap\fb$ such that
$\Omega\cap\partial\{u> 0\}$ has an exterior  unit normal
 $\nu(x_0)$ in the measure theoretic sense)
$u$ has the asymptotic development
\begin{equation}\label{asym-w}
u(x)=\lambda^*(x_0)\langle x-x_0,\nu(x_0)\rangle^-+o(|x-x_0|).
\end{equation}

\item For every $ x_0\in \Omega\cap\partial\{u>0\}$,
\begin{align*}
& \limsup_{\stackrel{x\to x_0}{u(x)>0}} |\nabla u(x)| \leq
\lambda^*(x_0).
\end{align*}

If there is a ball $B\subset\{u=0\}$ touching
$\Omega\cap\partial\{u>0\}$ at $x_0$ then,
$$\limsup_{\stackrel{x\to x_0}{u(x)>0}} \frac{u(x)}{\mbox{dist}(x,B)}\geq  \lambda^*(x_0). $$
\end{enumerate}
\end{defi}

\begin{defi}\label{nondegener} Let $v$ be a continuous nonnegative function
in a domain $\Omega\subset\mathbb{R}^N$. We say that $v$ is
nondegenerate at a point  $x_0\in \Omega\cap\{v=0\}$ if there
exist $c>0$, $\bar r_0>0$ such that one of the following conditions holds:
\begin{equation}\label{nond-prom-bol}
\fint_{B_r(x_0)} v\, dx\geq c r\quad \mbox{ for } 0<r\leq \bar r_0,
\end{equation}
\begin{equation}\label{nond-prom-casc}
\fint_{\partial B_r(x_0)} v\, dx\geq c r\quad \mbox{ for } 0<r\leq \bar r_0,
\end{equation}
\begin{equation}\label{nond-sup}
\sup_{B_r(x_0)} v\geq c r\quad \mbox{ for } 0<r\leq \bar r_0.
\end{equation}

\medskip

We say that $v$ is uniformly nondegenerate on a set
$\Gamma\subset\Omega\cap\{v=0\}$ in the sense of \eqref{nond-prom-bol} (resp. \eqref{nond-prom-casc}, \eqref{nond-sup}) if the constants $c$ and $\bar r_0$
in \eqref{nond-prom-bol} (resp.  \eqref{nond-prom-casc}, \eqref{nond-sup}) can be taken independent of the point $x_0\in\Gamma$.
\end{defi}

\begin{rema}\label{equiv-nondeg}
Assume that $v\ge 0$ is locally Lipschitz continuous in a domain $\Omega\subset\mathbb{R}^N$, $v\in W^{1,p(\cdot)}(\Omega)$ with $\Delta_{p(x)} v \ge  f \chi_{\{v >0\}}$,
where $f\in L^{\infty}(\Omega)$, $1<p_{\min}\le p(x)\le p_{\max}<\infty$ and $p(x)$ is
Lipschitz continuous. Then the three concepts of nondegeneracy in Definition \ref{nondegener} are equivalent
(for the idea of the proof, see Remark 3.1 in \cite{LW1}, where the case $p(x)\equiv 2$ and $f\equiv 0$ is treated).
\end{rema}

We will now derive some properties of the weak solutions.

\begin{lemm}\label{weak-radon}
If $u$ satisfies the hypothesis $(1)$ of Definition \ref{weak2} then $\lambda=\lambda_u:=
\Delta_{p(x)} u - f \chi_{\{u >0\}}$ is a nonnegative Radon measure with support on
$\Omega\cap\partial\{u>0\}$.
\end{lemm}
\begin{proof}
The proof follows as in the case $p(x)\equiv 2$, that was done in \cite{LW3}, Lemma 2.1.
\end{proof}

\begin{prop}\label{loc-lip}
Assume that $u$ satisfies  hypothesis $(1)$ of Definition \ref{weak2}. Assume moreover that $u\in L^\infty(\Omega)$,
$\|\nabla p\|_{L^{\infty}}\leq L$
and there exist  constants $C_0>0$, $\hat r_0>0$ such that if  $x\in \Omega\cap\partial\{u>0\}$, $B_r(x)\subset\Omega$ and $r\le \hat r_0$, then
$$\sup_{B_{r(x)}} u\leq C_0 r.$$
Then, $u$ is locally Lipschitz.
Moreover, for any $D\subset\subset
 \Omega$ the Lipschitz constant of $u$ in $D$ can be
estimated by a constant $C$ depending only on $N,  p_{\min}, p_{\max}, L,
{\rm dist}(D,\partial\Omega)$, $\|u\|_{L^\infty(\Omega)}$, $\|f\|_{L^\infty(\Omega)}$, $C_0$ and $\hat r_0$.
\end{prop}
\begin{proof}
We will find a constant $C$ such that $|\nabla u|\le C$ in
$D\cap\{u>0\}.$ Let $r_1 = {\rm dist}(D,\partial\Omega)$ and $y\in
D\cap\{u>0\}$ such that ${\rm dist}(y,\partial\{u>0\})<
\min\{\frac{\hat r_0}2, \frac{r_1}3,1\}$. Let $\bar
x\in\partial\{u>0\}$ such that $r={\rm
dist}(y,\partial\{u>0\})=|\bar x-y|$. Then $B_r(y)\subset
B_{2r}(\bar x)$ and thus,
$$\frac1r\sup_{B_{r(y)}} u \le \frac1r\sup_{B_{{2r}(\bar x)}} u\leq 2C_0.$$
We will show that there exists $\tilde C$ such that
$$|\nabla u(y)|\leq \tilde C\Big(1+\big(\frac{1}{r}\sup_{B_r(y)}
u\big)^{p_{\max}/p_{\min}}\Big).$$

In fact, let $v(z)=\frac1r u(y+rz)$. Then, $||v||_{L^\infty(B_1)}\le 2C_0$ and $\Delta_{{\bar p}(x)}v={\bar f}$ in $B_1$, with
${\bar p}(z)=p(y+rz)$, ${\bar f}(z)=rf(y+rz)$. There holds that $p_{\min}\le \bar p(x)\le p_{\max}$, $\|\nabla {\bar p}\|_{L^{\infty}}\leq L$ and
$\|{\bar f}\|_{L^\infty(B_1)}\le \|f\|_{L^\infty(\Omega)}$, if $0<r<1$. By the local results in \cite{Fan} it follows that $v\in C_{{\rm loc}}^{1,\alpha}(B_1)$ and
then, there exists $C_1>0$ such that $||\nabla v||_{C^{\alpha}(B_{1/2})}\le C_1$.
Therefore, if $z\in B_{1/2}(0)$
$$|\nabla v(0)|\leq C_2+|\nabla v(z)|,$$
and thus, if $x\in B_{r/2}(y)$,
$$|\nabla u(y)|\leq C_2+|\nabla u(x)|.$$
If $|\nabla u(y)|\leq 1$, the desired bound follows.
If $|\nabla u(y)|\geq 1$,  we get
$$|\nabla u(y)|^{p_{\min}}\leq |\nabla u(y)|^{p(x)} \leq C_3(1+|\nabla u(x)|^{p(x)}).$$
Integrating for $x\in B_{r/2}(y)$, we obtain
$$|\nabla u(y)|^{p_{\min}}\leq  C_3\Big(1+\pint_{B_{r/2}(y)}|\nabla u(x)|^{p(x)}\Big).$$
Applying Cacciopoli type inequality (see \cite{Fan}, Lemma 3.1, (3.5)) we have, for some constants $C_4$ and $R_0$ that, if $r\le R_0$ and
$\omega=\pint_{B_{r}(y)} u(x)$,
\begin{align*}|\nabla u(y)|^{p_{\min}}&\leq  C_4\Big(1+\pint_{B_{r}(y)}
\Big(\frac{|u(x)-\omega|}{r}\Big)^{p(x)}\Big)\\&\leq
C_4\Big(2+ \big(\frac{2}{r}\sup_{B_r(y)}
u\big)^{p_{\max}}\Big).
\end{align*}
This gives the result in case ${\rm dist}(y,\partial\{u>0\})<
R_1$, with $R_1=\min\{R_0, \frac{\hat r_0}2, \frac{r_1}3, 1\}$. If, on
the other hand, ${\rm dist}(y,\partial\{u>0\})\ge R_1$, the local
results of \cite{Fan} give
$$
|\nabla u(y)|\le \bar C,
$$
for a constant $\bar C$ depending on $N,  p_{\min}, p_{\max}, L$,
 $\|u\|_{L^\infty(\Omega)}$, $\|f\|_{L^\infty(\Omega)}$, $R_1$.
We thus obtain the desired estimate.
\end{proof}

\begin{lemm}\label{vale-viejo(2)} Assume that $u$ satisfies  hypotheses $(1)$ and $(2)$ of Definition
\ref{weak2}. For $D\subset\subset \Omega$ there are constants $0< \tilde c_{\min}\leq
\tilde C_{\max}$ and $\tilde r_0>0$ such that for balls $B_r(x)\subset D$ with $x\in
\partial \{u>0\}$ and $0<r\le \tilde r_0$
\begin{equation}\label{promedios}
\tilde c_{\min}\leq \frac{1}{r}\fint_{B_r(x)} u dx\, \leq \tilde C_{\max}.
\end{equation}
\end{lemm}
\begin{proof} The result follows from Proposition \ref{loc-lip}, Lemma \ref{weak-radon} and Remark \ref{equiv-nondeg}.
\end{proof}

\begin{lemm}\label{densprop2}
Assume that $u$ satisfies  hypotheses $(1)$ and $(2)$ of Definition
\ref{weak2}.

Then, for any domain $D\subset\subset \Omega$ there exist  constants
$c$ and $\bar r_0>0$, with $0<c<1$, depending on $||\nabla u||_{L^\infty(D)}$, $\|f\|_{L^\infty(D)}$ $r_0$,  $p_{\min}$, $p_{\max}$,
$||\nabla p||_{L^\infty(D)}$  and  $c_{\min}$, such that for every
$B_r\subset D$, centered at the free boundary with $0<r\le \bar r_0$ we have
$$\frac{|B_r\cap\{u>0\}|}{|B_r|}\geq c.$$
\end{lemm}
\begin{proof}
We first notice that, by Proposition \ref{loc-lip} and Lemma \ref{vale-viejo(2)}, $u$ is locally Lipschitz and \eqref{promedios} holds.
Let $B_r(x_0)\subset D$ with $x_0\in \partial \{u>0\}$. We observe that $u(x)\le r||\nabla u||_{L^\infty(D)}$ in $\{u>0\}\cap B_r(x_0)$. Therefore, for
$0<r\le \tilde r_0$
$$
\tilde c_{\min}\leq \frac{1}{r}\fint_{B_r(x_0)} u dx\, \leq ||\nabla u||_{L^\infty(D)}\frac{|B_r(x_0)\cap\{u>0\}|}{|B_r(x_0)|}.
$$
\end{proof}
\begin{rema}\label{medlebcero} Assume that $u$ satisfies  hypotheses $(1)$ and $(2)$ of Definition
\ref{weak2}. It follows from Lemma \ref{densprop2} that the free boundary has
Lebesgue measure zero.
\end{rema}

\begin{lemm} \label{equivmedida}
Assume that $u$ satisfies  hypotheses $(1)$ and $(2)$ of Definition
\ref{weak2}.

Then for any domain
$D\subset\subset \Omega$ there exist constants $\,c, C$ and $\bar r_0$
 depending on $N$, $p_{\min}$, $p_{\max}$, $||\nabla p||_{L^\infty(D)}$, $||f||_{L^\infty(D)}$, $||\nabla u||_{L^\infty(D)}$,  $c_{\min}$,
 $C_{\max}$ and $r_0$ such
that,  for every $B_r\subset D$
centered at the  free boundary, with $r\le \bar r_0$, we have
$$ c r^{N-1}\leq \int_{B_r}d\lambda \leq C r^{N-1}.$$
Here $\lambda=\lambda_u$ is as in Lemma \ref{weak-radon}.
\end{lemm}
\begin{proof}
Let $\xi\in C_0^\infty(\Omega)$, $\xi\geq 0$. Then,
$$\int_{\Omega} \xi d\lambda=-\int |\nabla u|^{p(x)-2}\nabla u \cdot \nabla
\xi\, dx -\int_{\{u>0\}} f  \xi \, dx.$$
Approximating $\chi_{B_r}$ from below by a
sequence  $\{\xi_n\}$ in $C_0^\infty(\Omega)$ such that $0\le\xi_n\le 1$,  $\xi_n=1$ in $B_{r-\frac{1}{n}}$ and
$|\nabla \xi_n|\leq C_N n$ and using that $u$ is locally Lipschitz, we
have that
$$
\begin{aligned}
-\int |\nabla u|^{p(x)-2}\nabla u \cdot \nabla
\xi_n\, dx -\int_{\{u>0\}} f  \xi_n \, dx&\leq C_0 n \big|B_{r}\setminus
B_{r-\frac{1}{n}}\big|+||f||_{{L^{\infty}}(D)}\big|B_{r}\big|\\
&\le C_1 r^{N-1},
\end{aligned}
$$
if $r\le 1$, with $C_0=C_0(p_{\max}, ||\nabla u||_{L{^{\infty}{(D)}}},N)$ and
$C_1=C_1(p_{\max}, ||\nabla u||_{L{^{\infty}{(D)}}}, N, ||f||_{L^{\infty}(D)}).$

Then, as
$$\int_{\Omega} \xi_n d\lambda\to \int_{B_r} d\lambda,$$
the bound from above holds.

Let us now prove the bound from below. Arguing by contradiction we assume that there exists a
sequence of functions $u_k$ satisfying hypotheses $(1)$ and $(2)$ of Definition
\ref{weak2} with power $p_k(x)$ and right hand side $f_k(x)$, with $p_{\min}\le p_k(x)\le p_{\max}$, $||\nabla p_k||_{L^\infty(D)}\le L_1$, $||f_k||_{L^\infty(D)}\le L_2$
and $||\nabla u_k||_{L^{\infty}(D)}\le L_0$, and balls $B_{r_k}(x_k)\subset D$, with $x_k\in
\partial\{u_k>0\}$ and $r_k\to 0$ with $\lambda_k=\Delta_{p_k(x)} u_k- f_k \chi_{\{u_k >0\}}$  satisfying that
$\int_{B_{r_k}(x_k)} d\lambda_k\le \varepsilon_k {r_k}^{N-1}$ with $\varepsilon_k \rightarrow 0$. Let $v_k(x)=\frac{u_k(x_k+r_k x)}{r_k}$.
As the $v_k's$ are
uniformly Lipschitz in $B_1(0)$ and $v_k(0)=0$, we can assume that $v_k\rightarrow v_0$
uniformly in $B_{1/2}$. We can also assume that $x_k\to x_0\in \overline D$.

We have $v_k\ge 0$  and $\Delta_{{\bar p}_k(x)}v_k={\bar f}_k$ in $B_1(0)\cap\{ v_k> 0\}$, with ${\bar p}_k(x)=p_k(x_k+r_k x)$,
${\bar f}_k(x)=r_k f_k(x_k+r_k x)$. We  can assume that ${\bar p}_k \to p_0\in\Bbb R$ uniformly on compact subsets of $B_1(0)$.

We claim that $\nabla v_k\to \nabla v_0$ a.e. in $B_{1/2}$. In fact, on one hand, by the interior H\"older gradient estimates, we have that
$\nabla v_k\to \nabla v_0$ uniformly on compact subsets of $\{v_0>0\}$.

On the other hand, if $B_r(\bar x)\subset\{v_0\equiv 0\}\cap B_{1/2}(0)$, then $B_{r/2}(\bar x)\cap\partial\{ v_k>0\}=\emptyset$ for large $k$ by the nondegeneracy.
So, either $B_{r/2}(\bar x)\subset\{v_k\equiv 0\}$ for a subsequence, or else  $v_k>0$ in  $B_{r/2}(\bar x)$ for large $k$. In any case,
$\nabla v_k\to \nabla v_0$ uniformly in $B_{r/4}(\bar x)$. Now observing that, with the same argument used in Remark \ref{medlebcero}, we get
that $|B_{1/2}(0)\cap\partial\{v_0>0\}|=0$, the claim follows.

Then,  for all $\xi \in C_0^{\infty}(B_{1/2})$, $\xi\geq 0$,
$$-\int_{B_{1/2}}|\nabla v_0|^{p_0-2}\nabla v_0 \cdot
\nabla \xi =\lim_{k \rightarrow \infty} \Big(-\int_{B_{1/2}}  |\nabla
v_k|^{{\bar p}_k(x)-2} \nabla v_k \cdot \nabla \xi - \int_{B_{1/2}}{\bar f}_k \xi\chi_{\{v_k >0\}} \Big).
$$
On the other hand, denoting $\varphi(y)=\xi(\frac{y-x_k}{r_k})$, we have
$$-\int_{B_{1/2}}  |\nabla
v_k|^{{\bar p}_k(x)-2} \nabla v_k \cdot \nabla \xi - \int_{B_{1/2}} {\bar f}_k \xi\chi_{\{v_k >0\}}  = \frac{1}{{r_k}^{N-1}}\int_{B_{{r_k}/2}(x_k)}
\varphi\, d\lambda_k \leq \|\varphi\|_{L^{\infty}(B_{{r_k}/2}(x_k))}  \ep_k \to 0.$$ Therefore $\Delta_{p_0}v_0=0$ in $B_{1/2}$. But $v_0\geq
0$ and  $v_0(0)=0$, so that  by the Harnack inequality we have
$v_0=0$ in $B_{1/2}$.

On the other hand, $0\in \partial\{v_k>0\}$, and by the
nondegeneracy, we have
$$\int_{B_{1/4}} v_k\geq c>0.$$
Thus,
$$\int_{B_{1/4}} v_0\geq c>0$$
which is a contradiction.
\end{proof}

\medskip

The next result gives a representation formula for weak solutions. We
will denote by ${\mathcal H}^{N-1}\lfloor\,\partial\{u>0\}$ the
measure ${\mathcal H}^{N-1}$ restricted to the set $\fb$.

\begin{theo}\label{rep-weak}
Assume that $u$ satisfies  hypotheses $(1)$ and $(2)$ of Definition
\ref{weak2}.  Then,

\item {1)} ${\mathcal
H}^{N-1}(D\cap\partial\{u>0\})<\infty$, for every
$D\subset\subset\Omega$.

\item {2)} There exist a borelian function $q_u$
defined on $\Omega\cap\partial\{u>0\}$  such that
$$
\Delta_{p(x)} u - f \chi_{\{u>0\}}=q_u\, {\mathcal H}^{N-1}\lfloor\,\partial\{u>0\}.
$$

 \item {3)} For every  $D\subset\subset \Omega$
there exist $C>0, c>0$ and $r_1>0$ such that
$$
c r^{N-1}\le{\mathcal H}^{N-1}(B_r(x_0)\cap\fb)\le C r^{N-1}
$$
for balls $B_r(x_0)\subset D$ with $x_0\in D\cap \fb$ and $0<r<r_1$ and, in addition,

\item {4)} $c\le q_u\le C$
 \ in \ $D\cap\partial\{u>0\}$.
\end{theo}
\begin{proof}
The result follows as Theorem 4.5 in \cite{AC}.
\end{proof}

\begin{rema}\label{perfin}
Assume that $u$ satisfies  hypotheses $(1)$ and $(2)$ of
Definition \ref{weak2}. It follows from Theorem \ref{rep-weak}
that the set $\Omega\cap\{u>0\}$ has finite perimeter locally in
$\Omega$ (see \cite{F} 4.5.11). That is, $\mu_{u}:=-\nabla \chi
_{\{u>0\}}$ is a Borel measure, and the total variation $|\mu_u|$
is a Radon measure. In this situation, we define the reduced
boundary as in \cite{F}, 4.5.5. (see also \cite{EG}) by,
$\partial_{\rm{red}}\{u>0\}:=\{x\in \Omega\cap\partial\{u>0\}/
|\nu_u(x)|=1\}$, where $\nu_u(x)$ is the unit vector with
\begin{equation}\label{carac}
\int_{B_r(x)}|\chi_{\{u>0\}} -\chi_{\{y/\langle y-x,\nu_u(x)\rangle<0\}}|=o(r^{N})
\end{equation}
for $r\to 0$, if
such a vector exists, and $\nu_u(x)=0$ otherwise. By the results
in \cite{F} Theorem 4.5.6, we have
$$\mu_u=\nu_u \H \lfloor \partial_{\rm{red}}\{u>0\}.$$
\end{rema}

\

We also have  the following result on blow up sequences

\begin{lemm}\label{propblowup} Assume that $u$ satisfies  hypotheses $(1)$ and $(2)$ of Definition
\ref{weak2}.
Let $B_{\rho_k}(x_k)\subset\Omega$ be a sequence of balls with
$\rho_k\to 0$, $x_k\to x_0\in \Omega$ and $u(x_k)=0$. Let us consider the blow-up sequence with respect to
$B_{\rho_k}(x_k)$. That is,
$$
u_k(x):=\frac{1}{\rho_k} u(x_k+\rho_k x).
$$
Then, there exists a blow-up
limit $u_0:\R^N\to\R$ such that, for a subsequence,

\begin{enumerate}

\item $u_k\to u_0$ in $C^\alpha_{\rm loc}(\R^N)$ for every $0<\alpha<1$,

\medskip

\item $\partial\{u_k>0\}\to \partial\{u_0>0\}$ locally in
Hausdorff distance,

\medskip
\item $\nabla u_k\to\nabla u_0$ uniformly on compact subsets of
$\{u_0>0\}$,

\medskip

\item $\nabla u_k\to\nabla u_0$ a.e. in $\R^N$,

\medskip

\item If $x_k\in \partial\{u>0\}$, then $0\in
\partial\{u_0>0\}$,
\medskip

\item $\Delta_{p(x_0)} u_0=0$ in $\{u_0>0\}$,

\medskip

\item $u_0$ is Lipschitz continuous and satisfies property (2) of Definition
\ref{weak2} in $\R^N$ with the same constants as $u$ in a ball $B_{\rho_0}(x_0)\subset\subset\Omega$ .
\end{enumerate}
\end{lemm}
\begin{proof}
The proof follows with similar ideas to those in \cite{AC}, 4.7 and \cite{ACF}, pp. 19-20. We here use that
$\Delta_{p_k(x)}u_k=f_k$ in $\{ u_k> 0\}$, where $p_k(x)=p(x_k+\rho_k x)$ and $f_k(x)=\rho_k f(x_k+\rho_k x)$
satisfy $p_k\to p(x_0)$ and $f_k \to 0$ uniformly on compact sets of $\R^N$. This implies that
$\nabla u_k$ are uniformly H\"older continuous on compact subsets of $\{u_0>0\}$. (Notice that some of these arguments were already employed
in the proof of Lemma \ref{equivmedida}).
\end{proof}

We will next prove an identification result for the function $q_u$ given in Theorem \ref{rep-weak}, which holds at points $x_0\in
\partial_{\rm{red}}\{u>0\}$ that are Lebesgue points of the function $q_u$ and are such that
\begin{equation}\label{limsupprop}
\limsup_{r\to 0}\frac{\H(\partial\{u>0\}\cap B(x_0,r))}{\H(B'(x_0,r))}\leq
1.
\end{equation}
(Here $B'(x_0,r)=\{x'\in \RR^{N-1}\,/\, |x'|<r\}$).

\smallskip

Notice that under our assumptions, $\H -\,a.e.$ point in
$\partial_{\rm{red}}\{u>0\}$ satisfies \eqref{limsupprop} (see
Theorem 4.5.6(2) in \cite{F}).

\begin{lemm}\label{glambda*}
Assume that $u$ satisfies  hypotheses $(1)$, $(2)$ and  $(3)$ of Definition
\ref{weak2}.  Then, $q_u(x_0)={\lambda^*(x_0)}^{p(x_0)-1}$ for
$\mathcal{H}^{N-1}$ a.e.  $x_0\in
\partial_{\rm{red}} \{u>0\}$.
\end{lemm}
\begin{proof}
If $u$ satisfies (3) of Definition \ref{weak2}, take $x_0\in
\partial_{\rm{red}}\{u>0\}$ such that
$$
u(x)=\lambda^*(x_0)\langle x-x_0,\nu(x_0)\rangle^-+o(|x-x_0|),
$$
where $\nu(x_0)$ is the exterior unit normal at $x_0$ in the measure theoretic sense. We assume $\nu(x_0)=e_N$.
Take
$\rho_k\to 0$ and $u_k(x)=\frac{1}{\rho_k}u(x_0+\rho_k x).$  If
$\xi\in C_0^{\infty}(\Omega)$ we have
$$-\int_{\{ u>0\}} |\nabla u|^{p(x)-2} \nabla u\cdot \nabla \xi\, dx -\int_{\{ u>0\}} f \xi\, dx =\int_{\partial \{u>0\}} q_u(x)  \xi d\H,$$ and
if we replace $\xi$ by $\xi_k(x)=\rho_k\xi(\frac{x-x_0}{\rho_k} )$
with $\xi\in C_0^{\infty}(B_R)$, $k\geq k_0$ and we change
variables, we obtain
$$
-\int_{\{ u_k>0\}} |\nabla u_k|^{p_k(x)-2}\nabla u_k\cdot \nabla \xi\, dx-\int_{\{ u_k>0\}} f_k \xi\, dx=\int_{\partial \{u_k>0\}} q_u(x_0+\rho_k x)
\xi d\H,
$$
where $p_k(x)=p(x_0+\rho_k x)$ and $f_k(x)=\rho_k f(x_0+\rho_k x)$.
{}From Lemma \ref{propblowup}, it follows that,  for a subsequence, $u_k\to u_0$ uniformly on compact sets of $\R^N$,
with $u_0(x)=\lambda^*(x_0)x_N^-$ and moreover,
$ |\nabla u_k|^{p_k(x)-2} \nabla u_k \to |\nabla u_0|^{p_0-2} \nabla u_0$ a.e. in $\RR^N$, with $p_0=p(x_0)$. Thus,
$$
-\int_{\{ u_k>0\}} |\nabla u_k|^{p_k(x)-2}\nabla u_k\cdot \nabla \xi\, dx-\int_{\{ u_k>0\}} f_k \xi\, dx\to
-\int_{\{x_N<0\}}|\nabla u_0|^{p_0-2} \nabla u_0 \cdot \nabla \xi\, dx.
$$

We now let
 $$\xi(x)=\min\big(2(1-|x_N|)^+,1\big)
 \eta(x_1,...,x_{N-1}),$$
 for $|x_N|\le 1$ and $\xi=0$ otherwise,
where $\eta\in C_0^{\infty}(B_r')$, (where $B'_{r}$ is a ball
$(N-1)$ dimensional with radius $r$) and $\eta\geq 0$. Then, if
$x_0$ is a Lebesgue point of $q_u$ satisfying \eqref{limsupprop},
 we proceed
as in \cite{AC}, p.121 and we get
\begin{equation}\label{convergencia}
\int_{\partial \{u_k>0\}} q_u(x_0+\rho_k x) \xi\, d\H\rightarrow
q_u(x_0)\int_{\{x_N=0\}} \xi \,d\H.
\end{equation}

As $\nabla u_0=-\lambda^*(x_0) e_N \chi_{\{x_N<0\}}$, it follows that
$$
\lambda^*(x_0)^{p_0-1}\int_{B_r'}\xi(x',0)\,d\mathcal{H}^{N-1}=
q_u(x_0)\int_{ B'_{r}} \xi(x',0)
\,d\mathcal{H}^{N-1}.
$$
Thus, we deduce that for $\mathcal{H}^{N-1}$-almost every point
$x_0 \in \partial_{\rm{red}}\{u>0\}$,
$q_u(x_0)={\lambda^*(x_0)}^{p(x_0)-1}$.
\end{proof}

\end{section}
%%%%%%%%%%end section weak solutions%%%%%%%%%%%%%%%%%%%%%%%%%%%%%%%%%%
%%%%%%%%%%%%%%%%%begin section flat free boundary points%%%%%%%%%%%%%%%%%%%%%%%
\begin{section}{Flat free boundary points}
\label{sect-flat-fbpoints}

\setcounter{equation}{0}

In this section we study the behavior of weak solutions to the free boundary problem \ref{bernoulli-px} near ``flat" free boundary points.

\smallskip

Throughout the section we assume, unless otherwise stated, that $f$ is bounded, $p(x)$ is Lipschitz continuous and $\lambda^*(x)$ is H\"older continuous.

\medskip

As in previous papers, we start by defining   the flatness
classes.

\begin{defi}\label{def-flat} Let $0<\sigma_1,\sigma_2\le 1$, $\tau>0$. We say that $u$ belongs to the class $F(\sigma_1,\sigma_2;\tau)$ in $B_\rho(x_0)$ in direction $\nu$ with power $p(x)$, slope $\lambda^*(x)$ and right hand side $f(x)$ if $u$ is a weak solution to the free boundary problem \ref{bernoulli-px} in $B_\rho(x_0)$, $x_0\in\partial\{u>0\}$ and
\begin{enumerate}
\item $u(x)=0$ if $\langle x-x_0,\nu\rangle\ge\sigma_1\rho$, $x\in
B_\rho(x_0)$,

\item \label{item-flat-2}$u(x)\ge-\lambda^*(x_0)\big(\langle
x-x_0,\nu\rangle+\sigma_2\rho\big)$ if $\langle
x-x_0,\nu\rangle\le -\sigma_2\rho$, $x\in B_\rho(x_0)$,

    \item $|\nabla u|\le \lambda^*(x_0)(1+\tau)$ in $B_\rho(x_0)$.

\end{enumerate}

\end{defi}

After a rotation and a translation we may assume that $x_0=0$ and
$\nu=e_N$.
 We will not explicitly mention the direction of flatness when $\nu=e_N$.

 We may further reduce the analysis to the unit ball by the following transformations:
\begin{equation}\label{reduction}
\bar u(x)=\frac{u(\rho x)}\rho,\qquad \bar p(x)=p(\rho x),\qquad
\bar\lstar(x)=\lstar(\rho x),\qquad \bar f(x)=\rho f(\rho x).
\end{equation}

Then, if $u\in F(\sigma_1,\sigma_2;\tau)$ in $B_\rho$ with power
$p$, slope $\lambda^*$ and right hand side $f$, there holds that
$\bar u\in F(\sigma_1,\sigma_2;\tau)$ in $B_1$ with power $\bar
p$, slope $\bar \lambda^*$ and right hand side $\bar f$.

Observe that, if $1<p_{\min}\le p(x)\le p_{\max}<\infty$, $0<\lone\le
\lstar(x)\le \ltwo<\infty$, $p\in Lip$ with $|\nabla p|\le
L_1$, $\lstar\in C^{\alpha^*}$ with
$[\lstar]_{C^{\alpha^*}(B_\rho)}\le C^*$ and $f\in
L^\infty(B_\rho)$ with $|f|\le L_2$, there holds that $\bar p$,
$\bar\lstar$ and $\bar f$ are in similar spaces in $B_1$ and
$1<p_{\min}\le \bar p(x)\le p_{\max}<\infty$, $0<\lone\le \bar\lstar(x)\le
\ltwo<\infty$, $|\nabla\bar p|\le L_1\rho$, $|\bar f|\le L_2\rho$
and $[\bar\lstar]_{C^{\alpha^*}(B_1)}\le C^*\rho^{\alpha^*}$.

\bigskip

The first lemma states that, if $u$ vanishes for $x_N\ge\sigma$,
there holds that, in a smaller ball, $u$ is above a hyperplane for
$x_N\le -\ep$.
\begin{lemm}\label{lema 1} Let $p\in Lip(B_1)$, $\lstar\in C^{\alpha^*}(B_1)$, $f\in L^\infty(B_1)$ with $|\nabla p|\le L_1\rho$,
$|f|\le L_2\rho$, $[\lstar]_{C^{\alpha^*}(B_1)}\le C^*\rho^{\alpha^*}$ and $C^*\rho^{\alpha^*}\le\lstar(0)\sigma$. Let $u\in F(\sigma,1;\sigma)$ in $B_1$
with power $p$, slope $\lstar$ and rhs $f$.

Let $0<\ep\le1/2$ and $\frac12\le R<1$. There exists
$\sigma_0=\sigma_0(\ep, N, R, p_{\min}, p_{\max}, \lone, \ltwo, L_1, L_2, C^*)$ such
that if   $\sigma\le \sigma_0$ there holds that $u\in
F(\sigma/R,\ep;\sigma)$ in $B_{R}$ with the same power, slope and
rhs.

\end{lemm}
\begin{proof} We follow the construction of \cite{ACF} with the variation of  \cite{DP1}. In this paper, we consider an arbitrary $R$ instead of $R=1/2$ in order to pursue the argument in the next steps.

Let $R'=R+(1-R)/4$. As in these  papers, we will prove that, for every
$0<r\le (1-R)/8$  there exists
$\sigma_0=\sigma_0(r,R,p_{\min},p_{\max},\lone,\ltwo, L_1,L_2,C^*)$ such
that for $\sigma\le\sigma_0$,
\begin{equation}\label{ineq-xi}
u(\xi)\ge\lstar(0)[-\xi_N-4r] \quad\mbox{for}\quad \xi\in\partial
B_{R'}\mbox{ with }\xi_N\le -\frac{(1-R)}{4}.
\end{equation}
Then, integrating along vertical lines a distance at most $R'$ and
using that $|\nabla u|\le \lstar(0)(1+\sigma)$, we get
\[
\begin{aligned}
u(\xi',\xi_N+\alpha)&\ge u(\xi)-\lstar(0)(1+\sigma)\alpha\\
&\ge \lstar(0)\big[-(\xi_N+\alpha)-4r-R'\sigma\big]\\
&\ge \lstar(0)\big[-(\xi_N+\alpha)-\ep R\big]
\end{aligned}
\]
if $0\le\alpha\le R'$, $r=\min\{\frac{R\ep}{8},\frac{1-R}{8}\}$
and $\sigma\le \min\{\frac{R\ep}{R+1}, \sigma_0\}$.

This implies that, for $|x|<R$, $x_N\le -R\ep$,
\[
u(x)\ge -\lstar(0)\big(x_N+R\ep\big).
\]
So that $u\in F(\sigma/R,\ep;\sigma)$ in $B_{R}$ with power $p$, slope $\lstar$ and rhs $f$, and the lemma
will be proved.

\medskip

In order to prove \eqref{ineq-xi}, we will show that, once we fix
$0<r\le\frac{(1-R)}{8}$ there exists $\kappa>0$ such that, for
every $\xi\in \partial B_{R'}$ with $\xi_N\le-(1-R)/4$, there
exists $x_\xi\in\partial B_r(\xi)$ such that
\begin{equation}\label{ineq-xxi}
u(x_\xi)\ge-\lstar(0)(1-\kappa\sigma) {x_\xi}_N.
\end{equation}
Then, by using again that $|\nabla u|\le \lstar(0)(1+\sigma)$,
\[\begin{aligned}
u(\xi)&\ge u(x_\xi)-\lstar(0)(1+\sigma)r\ge \lstar(0)[-(1-\kappa\sigma){x_\xi}_N-(1+\sigma)r]\\
&\ge \lstar(0)[-\xi_N-r-\kappa\sigma-2r]\ge\lstar(0)[-\xi_N-4r]
\end{aligned}
\]
 if $\sigma\le\frac r\kappa $, that is, we get \eqref{ineq-xi}.

The existence of a point $x_\xi$ satisfying \eqref{ineq-xxi} is
done by assuming that such a point does not exist and getting a
contradiction if $\kappa$ is large depending on $r,R$ and the
constants in the structure conditions. The inequality that will
allow to get this contradiction will be achieved if $\sigma$ is
small depending on  the same parameters. Such inequality comes
from the construction of two barriers in the following way:

Let $\eta\in C_0^\infty(B_1')$ given by
\[
\eta(y)=\begin{cases}{\rm exp}\Big(-\frac{9|y|^2}{1-9|y|^2}\Big)\quad&\mbox{if}\quad|y|<\frac13,\\
0\quad&\mbox{if}\quad|y|\ge\frac13.
\end{cases}
\]

Let $s\ge0$ be maximal such that
\[
B_1\cap\{u>0\}\subset D:=\{x\in B_1:x_N<\sigma-s\eta(x')\}.
\]

Then, as $0\in\partial\{u>0\}$ there holds that $s\le \sigma$.

First, we let $v\in W^{1,p(\cdot)}(D\setminus \overline{B_r(\xi)})$ be the solution to
\begin{equation}\label{eq-v}
\begin{cases}
\Delta_{p(x)}v=-L_2\rho\quad&\mbox{in}\quad D\setminus \overline{B_r(\xi)},\\
v=0\quad&\mbox{on}\quad \partial D\cap B_1,\\
v=\lstar(0)(1+\sigma)(\sigma-x_N)\quad&\mbox{on}\quad\partial D\setminus B_1,\\
v=-\lstar(0)(1-\kappa\sigma)x_N\quad&\mbox{on}\quad \partial
B_r(\xi).
\end{cases}
\end{equation}
Since the boundary datum coincides with $\lstar(0)(1+\sigma)(\sigma-x_N-s\eta(x'))$ on $\partial D$, it has an extension 
$\phi\in W^{1,\infty}(D\setminus \overline{B_r(\xi)})$ and therefore the solution $v$ exists by a minimization argument
in $\phi+W_0^{1,p(\cdot)}(D\setminus \overline{B_r(\xi)})$.

As we are assuming that \eqref{ineq-xxi} does not hold for
any $x_\xi\in\partial B_r(\xi)$ and, since $u=0$ if $x\in\partial
D\cap B_1$ and $|\nabla u|\le \lstar(0)(1+\sigma)$, there holds that
$u\le v$ on $\partial (D\setminus \overline{B_r(\xi)})$. Now, recalling Lemma \ref{weak-radon}, we get 
$\Delta_{p(x)} u \ge f \chi_{\{u >0\}}\ge -L_2\rho$, then comparison of weak sub- and super-solutions gives
\[
u\le v\quad\mbox{in}\quad D\setminus \overline{B_r(\xi)}.
\]

Now, let $z\in\partial D\cap\partial\{u>0\}\cap\{|z'|<1/3\}$.
Then, there exists a ball $B$ contained in $\{u=0\}$ such that
$z\in\partial B$. By the definition of weak solution and, since
$\lstar(z)\ge\lstar(0)-C^*\rho^{\alpha^*}|z|^{\alpha^*}\ge
\lstar(0)(1-\sigma)$, we deduce that
\begin{equation}\label{ineq-lstar}
\lstar(0)(1-\sigma)\le \lstar(z)\le\limsup_{\stackrel{x\to
z}{u(x)>0}} \frac{u(x)}{\mbox{dist}(x,B)}\le |\nabla v(z)|.
\end{equation}

We will get a contradiction once we find a barrier from above for
$v$ in the form $w=v_1-\kappa\sigma v_2$ with $|\nabla v_1|\le
\lstar(0)(1+C_3\sigma)$, $|\nabla v_2|\ge c\lstar(0)>0$, $v_1>0$,
$v_2>0$ close to $z$ and $v_1=v_2=0$ on $\partial D\cap B_1$ close
to $z$. In fact, if such a barrier $w$ exists, by
\eqref{ineq-lstar} there holds that
\[
\lstar(0)(1-\sigma)\le|\nabla v(z)|\le |\nabla w(z)|=|\nabla
v_1(z)|-\kappa\sigma|\nabla v_2(z)|\le
\lstar(0)\big[1+C_3\sigma-c\kappa\sigma\big]
\]
and this is a contradiction  if $\kappa$ is large depending only
on $C_3$ and $c$. Since the constants $C_3$ and $c$ will depend
only on $r,R,p_{\min},p_{\max},\lone,\ltwo, L_1,L_2$ and $C^*$, the lemma
will be proved.

As in \cite{DP1} and \cite{FBMW}, the idea of the construction of
$v_1$ and $v_2$ is that they will be such that $w=v_1-\kappa\sigma
v_2$ will satisfy
\begin{equation}\label{gradiente w}
\frac{\lstar(0)}2\le |\nabla w|\le 2\lstar(0)
\end{equation}
if $\sigma$ is small depending on those constants. Then,
\[
\Delta_{p(x)}w=|\nabla
w|^{p(x)-2}\Big[\sum_{ij}b_{ij}(x)w_{x_ix_j}+\sum_jb_j(x)w_{x_j}\Big]
\]
with $b_{ij}=\delta_{ij}+(p(x)-2)\frac{w_{x_i}w_{x_j}}{|\nabla
w|^2}$ and $b_j=p_{x_j}\log|\nabla w|$. There holds that
\begin{equation}\label{bij}
\beta_1|\nu|^2\le \sum_{ij}b_{ij}\nu_i\nu_j\le
\beta_2|\nu|^2\quad\forall \nu\in\R^N
\end{equation}
with $\beta_1=\min\{1,p_{\min}-1\}$, $\beta_2=\max\{1,p_{\max}-1\}$ and,
with  $\Lambda=\max\{|\log\lone|,|\log\ltwo|\}+\log 2$,
$b=(b_1,\cdots,b_N)$,
\begin{equation}\label{bj}
|b|\le \Lambda L_1\rho\le\frac{\Lambda L_1\ltwo}{C^*}\sigma=
C_0\sigma,
\end{equation}
if $\sigma\le \frac{C^*}{\ltwo}$, with $C_0=\frac{\Lambda
L_1\ltwo}{C^*}$.

Thus, the idea is to construct $v_1$ in such a way that
\[
\frac23\lstar(0)\le |\nabla v_1|\le \frac32\lstar(0)
\]
and
\[
{\mathcal T}v_1\le -S^{-1}L_2\frac{\ltwo}{C^*}\sigma=-M\sigma \quad {\rm in } \quad D,
\]
with
$S=\min\{\big(\frac{\lone}2\big)^{p_{\min}-2},\big(\frac{\lone}2\big)^{p_{\max}-2},(2\ltwo)^{p_{\min}-2},
(2\ltwo)^{p_{\max}-2}\}$ for any operator
\[
{\mathcal
T}=\sum_{ij}b_{ij}(x)\partial_{x_ix_j}+\sum_jb_j(x)\partial_{x_j}
\]
with $\{b_{ij}\}$  satisfying \eqref{bij} with
$\beta_1=\min\{1,p_{\min}-1\}$, $\beta_2=\max\{1,p_{\max}-1\}$  and
$\{b_j\}$ satisfying
\[
|b|\le C_0\sigma
\]
with $C_0$ the constant in \eqref{bj}.

Then, $v_2$ will be a function satisfying
\[
{\mathcal T}v_2\ge 0\quad\mbox{in}\quad\widetilde D\setminus
B_r(\xi)
\]
for any such an operator $\mathcal T$ with
\[
0<c\lstar(0)\le |\nabla v_2|\le C\lstar(0)
\]
for some constants $c,C$ depending only on $R,r$. Here $\widetilde
D$ is a smooth domain contained in $D$ and containing $D\setminus
B_{(1-R)/10}(\partial {B_1'}\times\{0\})$. In this way, once we
fix $\kappa>0$ there holds that  $w$ satisfies \eqref{gradiente w}
if $\sigma$ is small and therefore,
\[
\Delta_{p(x)}w\le -L_2\rho=\Delta_{p(x)}v \quad\mbox{in}\quad
\widetilde D\setminus B_r(\xi).
\]

The functions $v_1$ and $v_2$ are also constructed in such a way that $w\ge v$
on $\partial\big(\widetilde D\setminus B_r(\xi)\big)$.

As in the previously cited papers, we let
\[
d_1(x)=-x_N+\sigma-s\eta(x')\quad\mbox{and}\quad
v_1(x)=\lstar(0)\frac{\gamma_1}{\mu_1}\big(1-e^{-\mu_1d_1(x)}\big)\quad\mbox{in}\quad
D
\]
with $\mu_1=C_1\sigma$ and $\gamma_1=1+C_2\sigma$. Then, $|\nabla
v_1|\le \lstar(0) (1+C\sigma)(1+C_2\sigma)$ with $C$ depending
only on $\eta$ (in particular, $|\nabla v_1|\le
\lstar(0)(1+C_3\sigma)$ with $C_3$ depending only on $C_2$ and
$\eta$). Moreover, $D_{x_ix_j}v_1=\lstar(0)\gamma_1e^{-\mu_1d_1}
\big[D_{x_ix_j}d_1-\mu_1{d_1}_{x_i}{d_1}_{x_j}\big]$. Thus,
\[
\begin{aligned}
{\mathcal T}v_1&\le \gamma_1e^{-\mu_1d_1}\Big[ N^2\ltwo \beta_2\|D^2\eta\|_{L^{\infty}}\sigma-\lone\beta_1\mu_1+ \ltwo C_0(1+C_3\sigma)\sigma\Big]\\
&\le \big[2 N^2\ltwo\beta_2\|D^2\eta\|_{L^{\infty}}+4\ltwo C_0-e^{-2}C_1\lone\beta_1]\sigma\\
&\le -M\sigma
\end{aligned}
\]
if $\sigma\le \sigma(C_1,C_2,C_3)$ and $C_1\ge
C_1(\lone,\ltwo,\beta_1,\beta_2,C_0, M)$.  $C_1$ is fixed from now
on.

On the other hand,
\begin{equation}\label{nabla v 1}
\frac23\lstar(0)\le
\lstar(0)(1+C_2\sigma)e^{-C_1\sigma(1+\sigma)}\le |\nabla v_1|\le
\lstar(0)(1+C_3\sigma)\le \frac32\lstar(0)
\end{equation}
if $\sigma\le \sigma(C_1,C_2,C_3)$.

\smallskip

The constant $C_2$ (and therefore also $C_3$) will be fixed now
in order to guaranty that $w\ge v$ on the boundary of $D\setminus
B_r(\xi)$.

\medskip

First, on $\partial D\cap B_1$ we have $v_1=0$.

\medskip

Observe that
\[
v_1(x)\ge\lstar(0)(1+C_2\sigma)e^{-2C_1\sigma
}d_1\ge\lstar(0)\big(1+\frac{C_2}2\sigma\big)d_1\ge
\lstar(0)(1+4\sigma)d_1
\]
if $C_2\ge8$ and $\sigma\le\sigma(C_1,C_2)$.

\smallskip

Now, on $\partial D\setminus B_1$ we consider two cases:

\smallskip

\noindent (a) $|x'|\ge\frac13$. Then, $\eta(x')=0$  and
$d_1=\sigma-x_N$. Thus, \[
v_1(x)\ge\lstar(0)(1+\sigma)(\sigma-x_N).
\]

\smallskip

\noindent (b) $|x'|<\frac13$. Then, $|x_N|>\sqrt{\frac23}$ and
\[\begin{aligned}
v_1(x)&\ge\lstar(0)\big(1+4\sigma\big)(\sigma-x_N-s\eta(x'))\\
&\ge
\lstar(0)(1+\sigma)(\sigma-x_N)+\lstar(0)\big[3(\sigma-x_N)-(1+4\sigma)\big]\sigma\\
&\ge\lstar(0)(1+\sigma)(\sigma-x_N)+\lstar(0)\big[\sqrt 6-(1+4\sigma)\big]\sigma\\
&\ge \lstar(0)(1+\sigma)(\sigma-x_N)
\end{aligned}
\]
if $C_2\ge8$, $\sigma\le\sigma(C_1,C_2)$ and $\sqrt
6-(1+4\sigma)\ge0$.

\medskip

Finally, if $x\in\partial B_r(\xi)$ and, since $r\le
\frac{(1-R)}{8}$, there holds that $x_N<0$, so that
\[
\begin{aligned}
v_1(x)&\ge \lstar(0)(1+4\sigma)(\sigma-x_N-s\eta(x'))\\
&=\lstar(0)\big[-x_N+(1+4\sigma)(\sigma-s\eta(x'))-4\sigma x_N\big]\\
&\ge-\lstar(0)x_N.
\end{aligned}
\]

Therefore, we can fix $C_2=8$ for our construction of $v_1$.

Now, we construct $v_2$ in $\widetilde D\setminus B_r(\xi)$ with
$\widetilde D$ as described above. We take $d_2$ such that
\[
d_2\in C^2(\overline{\widetilde D\setminus B_r(\xi)}),\quad
d_2=0\mbox{ on }\partial\widetilde D,\quad 0\le d_2\le 1\mbox{ in
}\widetilde D\setminus B_r(\xi)
\]
and, moreover
\[
0<\tilde c\le |\nabla d_2|\le \tilde
C\quad\mbox{in}\quad\widetilde D\setminus B_r(\xi)
\]
with $\tilde C,\tilde c$ depending only on $r,R$.

Then, we take
\[
v_2(x)=\lstar(0)\frac{\gamma_2}{\mu_2}\big(e^{\mu_2d_2(x)}-1\big).
\]

First, we fix $\mu_2$. Then, $\gamma_2$ is fixed so that
$v_2\le\frac{(1-R)}{8}\lstar(0)$, that is,
\[
\gamma_2=\frac{(1-R)}8\frac{\mu_2}{(e^{\mu_2}-1)}.
\]

Thus, there exist constants depending only on $\tilde c,\tilde
C,\mu_2,R$ such that
\[
0<c\lstar(0)\le |\nabla v_2|\le C\lstar(0).
\]

Now, we  fix $\mu_2$ so that ${\mathcal T}v_2\ge 0$ in $\widetilde
D\setminus B_r(\xi)$ for any operator ${\mathcal T}$  as above.

There holds
\[
{\mathcal T}  v_2\ge \gamma_2\big[\mu_2\lone\beta_1\tilde
c^2-\beta_2\ltwo\|D^2 d_2\|_{L^\infty}-\tilde CC_0\sigma\ltwo\big]\ge0
\]
if $\mu_2\ge \mu_2(\lone,\ltwo,\beta_1,\beta_2, \tilde c,\tilde C,
C_0)$. (Recall that $\tilde c$ and $\tilde C$ depend only on
$r,R$).

Now, in order to finish our proof we need to see that $w=v_1-\kappa\sigma v_2\ge v$ in $\widetilde D\setminus B_r(\xi)$. For this purpose, it only
remains to show that the inequality holds on $\partial B_r(\xi)$, that is, we have to prove that
\[
w(x)=v_1(x)-\kappa\sigma v_2(x)\ge
-\lstar(0)(1-\kappa\sigma)x_N\quad\mbox{on}\quad\partial B_r(\xi).
\]

Recall that $v_2\le \frac{(1-R)}{8}\lstar(0)$. Thus,
\[
w(x)=v_1(x)-\kappa\sigma v_2(x)\ge
\lstar(0)(-x_N-\frac{(1-R)}{8}\kappa\sigma)\ge-\lstar(0)(1-\kappa\sigma)x_N
\]
since $x_N\le -\frac{(1-R)}{8}$ for $x\in\partial B_r(\xi)$.

And we get a contradiction as discussed above.
\end{proof}

\medskip

The following lemma gives a control of the gradient of $u$  from
below on compact sets of $B_1^-$.
\begin{lemm}\label{lema 2} Let $p, \lstar, f, \rho, u$ as in Lemma \ref{lema 1}. For every $\ep,\delta>0$, $\frac12\le R<1$, there exists
$\sigma_0$ depending  on $\ep, N, \delta, R, p_{\min}, p_{\max},
\lone, \ltwo, L_1, L_2, C^*$ such that, if $\sigma\le\sigma_0$
there holds that
\[
|\nabla u|\ge \lstar(0)(1-\delta)\quad\mbox{in}\quad
B_{R}\cap\{x_N\le-\ep\}.
\]
\end{lemm}
\begin{proof} The proof is entirely similar to the one of Lemma 6.6 in \cite{DP1}. Let $R<R'<1$. As in \cite{DP1} we use a contradiction argument. In our case  by Lemma \ref{lema 1}, we have that the functions $u_k\in F(\frac 1{k},1;\frac1k)$ in $B_{1}$ satisfy
\[
\Delta_{p_k(x)}u_k=f_k\quad\mbox{in}\quad\mathcal K\subset\subset
B_{R'}^-,
\]
if $k$ is large depending on $\mathcal K$. Here $|f_k|\le
L_2\rho_k$,  $1<p_{\min}\le p_k(x)\le p_{\max}<\infty$, $|\nabla p_k|\le
L_1\rho_k$ and $C^*{\rho_k}^{\alpha^*}\le \frac{\lstar_k(0)}k$. Thus, by the
regularity estimates in \cite{Fan}, for a subsequence, $\nabla
u_k$ converges uniformly on compact subsets of $B_{R'}^-$. And the
proof follows as in \cite{DP1}.
\end{proof}

Now we can prove one of the main results that states that,
flatness to the right ($u$ vanishing for $x_N\ge\sigma$) implies
flatness to the left in a smaller ball.
\begin{prop}\label{prop-flatness}  Let $p, \lstar, f, \rho, u$ as in Lemma \ref{lema 1}. Let $1/2\le R<1$. There exist \linebreak
$\sigma_0=\sigma_0(N, R, p_{\min}, p_{\max}, \lone, \ltwo, L_1, L_2, C^*)$, $C_0=C_0(N, R, p_{\min}, p_{\max}, \lone, \ltwo, L_1, L_2, C^*)$ such that, if $\sigma\le \sigma_0$ there holds that $u\in F(\sigma/R,C_0\sigma;\sigma)$ in $B_{R}$ with the same power, slope and rhs.
\end{prop}
\begin{proof} The proof follows as the one of Theorem 6.3 in \cite{DP1}. We let $R'=R+(1-R)/4$ and $R''=R+(1-R)/2$. In our case, since $|\nabla u|\ge \frac{\lstar(0)}2$ in $\overline{B_{R''}}\cap\{x_N\le -(1-R)/{8}\}$ if $\sigma$ is small and $|\nabla u|\le 2\lstar(0)$, there holds that $u$ satisfies
\[
{\mathcal T}u=|\nabla u|^{2-p(x)}f(x)\quad\mbox{in}\quad
B_{R''}\cap\{x_N< -(1-R)/{8}\}
\]
for an operator as the one considered in Lemma \ref{lema 1}.

Then, as in \cite{DP1} (see also \cite{ACF}) we take
\[
w(x)=\lstar(0)(1+\sigma)(\sigma-x_N)- u(x)
\]
that satisfies
\[
{\mathcal T}w=-\lstar(0)(1+\sigma)b_N-|\nabla
u|^{2-p(x)}f(x)\quad\mbox{in}\quad B_{R''}\cap\{x_N<
-\frac{(1-R)}{8}\}
\]
and, using that $w\ge0$ in $B_{1}\cap\{x_N\le\sigma\}$, taking
$\xi\in\partial B_{R'}\cap\{x_N\le -(1-R)/4\}$, applying Harnack
inequality in $B_{(1-R)/8}(\xi)$ and using that the right hand
side is bounded by $C\sigma$ for a constant $C$ depending only on
$R,p_{\min},p_{\max},\lone,\ltwo, L_1, L_2$ and $C^*$ we get, as in
\cite{ACF,DP1},
\[
w(\xi)\le \widetilde C\lstar(0)\sigma.
\]

Then, the proof follows as in \cite{DP1}.
\end{proof}

Finally, we can improve on the control of the gradient.
\begin{lemm}\label{lema 3}Let $p, \lstar, f, \rho, u$ as in Lemma \ref{lema 1}. For every  $1/2\le R<1$, $0<\delta<1$ there exists $\sigma_{\delta,R}$ and $C_{\delta,R}$
depending also on $N, p_{\min}, p_{\max}, \lone, \ltwo, L_1, L_2, C^*$ such that, if $\sigma\le\sigma_{\delta,R}$ there holds that
\[
|\nabla u|\ge\lstar(0)(1-\delta)\quad\mbox{in}\quad
B_{R}\cap\{x_N\le -C_{\delta,R}\sigma\}.
\]
\end{lemm}
\begin{proof} It follows exactly as the proof of Theorem 6.4 in \cite{DP1}.

Observe that the scalings $\bar p_k(x)=p_k(y_k+2d_k x)$, $\bar
{\lstar_k}(x)=\lstar_k(y_k+2d_k x)$ and $\bar
f_k(x)=2d_kf_k(y_k+2d_k x)$ satisfy the same structure conditions
as the functions $p_k$, $\lstar_k$ and $f_k$ that are independent
of $k$ in the contradiction argument.
\end{proof}

Now, in order to improve the flatness in some possibly  new
direction we perform a non-ho\-mo\-ge\-neous blow up.

\begin{lemm}\label{lema 4} Let $u_k\in F(\sigma_k,\sigma_k;\tau_k)$ in $B_1$ with power $p_k$, slope $\lstar_k$ and rhs $f_k$ such
that $1<p_{\min}\le p_k(x)\le p_{\max}<\infty$, $0<\lone\le\lstar_k(x)\le\ltwo<\infty$, $|\nabla p_k|\le L_1\rho_k$, $|f_k|\le L_2\rho_k$, $[\lstar_k]_{C^{\alpha^*}}\le C^*\rho_k^{\alpha^*}$ with $C^*\rho_k^{\alpha^*}\le\lstar_k(0)\tau_k$, $\sigma_k\to0$ and $\frac{\tau_k}{\sigma_k^2}\to0$ as $k\to\infty$.

For $y\in B_1'$, let
\[
\begin{aligned}
& F_k^+(y):=\sup\{h\,/\,(y,\sigma_k h)\in\partial\{u_k>0\}\},\\
& F_k^-(y):=\inf\{h\,/\,(y,\sigma_k h)\in\partial\{u_k>0\}\}.
\end{aligned}
\]

Then, for a subsequence,
\[\begin{aligned}
(1)\ &F(y):=\limsup_{\stackrel{z\to y}{k\to\infty}}F_k^+(z)=\liminf_{\stackrel{z\to y}{k\to\infty}}F_k^-(z) \mbox{ for every }y\in B_1'.\\
    &Moreover, F_k^+\to F,\  F_k^-\to F\  \mbox{uniformly},\   F \mbox{ is continuous}, \ F(0)=0 \mbox{ and } |F|\le 1.\\
\ \\
(2)\  &F \mbox{ is subharmonic}.
\end{aligned}\]
\end{lemm}
\begin{proof}(1) is proved exactly as in Lemma 7.3 in \cite{AC}.

In order to prove (2), we take $g$ a
harmonic function in a neighborhood of $B'_r(y_0)\subset\subset B'_1$
with $g>F$ on $\partial B'_r(y_0)$ and $g(y_0)<F(y_0)$ and get a
contradiction. We define the sets $Z_+(\phi),Z_-(\phi)$ and $Z_0(\phi)$ as in
the previous papers. That is,
\[Z:=B'_r(y_0)\times\R, \qquad
Z_+(\phi):=\{(y,h)\in Z \,/\, h>\phi(y)\}\]
and corresponding definitions for $Z_-(\phi),Z_0(\phi)$.

Observe that we may assume that ${\mathcal H}^{N-1}\big(Z_0(\sigma_kg)\cap\partial\{u_k>0\}\big)=0$. If not, we replace $g$ by $g+c_0$ for some small enough constant $c_0$.

In fact, let $c_1>0$ small such that $g(y_0)<g(y_0) + c <F(y_0)$ for $0<c<c_1$. Since by {Theorem \ref{rep-weak}}
${\mathcal H}^{N-1}(D\cap\partial\{u_k>0\})<\infty$ for every $D\subset\subset B_1$, we see that
\[|\{(y,h)\in Z \,/\,  \sigma_kg(y)<h<\sigma_k(g(y)+c_1)\}\cap\partial\{u_k>0\}|=0,\]
which implies that $\int_0^{c_1} H_k(c) dc=0$, for $H_k(c)={\mathcal H}^{N-1}\big(Z_0(\sigma_k(g+c))\cap\partial\{u_k>0\}\big)$.
Then, we can take $c_0\in(0,c_1)$ such that $H_k(c_0)=0$ for every $k$, and now replacing $g$ by $g+c_0$ we have 
${\mathcal H}^{N-1}\big(Z_0(\sigma_kg)\cap\partial\{u_k>0\}\big)=0$. 

In the following we denote $Z_+=Z_+(\sigma_kg)$ and similarly $Z_-$ and $Z_0$.

 Now, by using the representation formula (Theorem \ref{rep-weak}) and proceeding as in \cite{AC}, Lemma 7.5, we get
\[
\int_{\{u_k>0\}\cap Z_0}|\nabla u_k|^{p_k(x)-2}\nabla
u_k\cdot\nu\,d{\mathcal H}^{N-1}=\int_{\partial\{u_k>0\}\cap
Z_+}q_{u_k}\,d{\mathcal H}^{N-1}+\int_{\{u_k>0\}\cap Z_+}f_k\,dx.
\]

Since $q_{u_k}\ge0$ and $q_{u_k}(x)=\lstar_k(x)^{p_k(x)-1}$ \ \
${\mathcal H}^{N-1}-a.e.$ on $\partial_{\rm{red}}\{u_k>0\}$,
\begin{equation}\label{eq-1}\begin{aligned}
&\int_{\partial\{u_k>0\}\cap Z_+}q_{u_k}\,d{\mathcal H}^{N-1}\ge\int_{\partial_{\rm{red}}\{u_k>0\}\cap Z_+}{\lstar_k}^{p_k-1}\,d{\mathcal H}^{N-1}\\
&\ \ \
\ge\min\Big\{\big(\lstar_k(0)(1-C^{**}\rho_k^{\alpha^*})\big)^{p_k^+-1},\big(\lstar_k(0)(1-C^{**}\rho_k^{\alpha^*})\big)^{p_k^--1}\Big\}
{\mathcal H}^{N-1}\big(\partial_{\rm{red}}\{u_k>0\}\cap Z_+\big)
\end{aligned}
\end{equation}
where $C^{**}=\frac{C^{*}}{\lone}$, $p_k^+=\sup_{B_{1}}p_k$ and
$p_k^-=\inf_{B_{1}}p_k$. Recall that $p_k^+-p_k^-\le L_1\rho_k$.

On the other hand,
\begin{equation}\label{eq-2}
\int_{\{u_k>0\}\cap Z_+}f_k\,dx\ge -L_2\rho_k\big|\{u_k>0\}\cap
Z_+\big|.
\end{equation}

Finally,
\begin{equation}\label{eq-3}\begin{aligned}
&\int_{\{u_k>0\}\cap Z_0}|\nabla u_k|^{p_k(x)-2}\nabla u_k\cdot\nu\,d{\mathcal H}^{N-1}\\
&\ \ \ \le\max\Big\{\big(\lstar_k(0)(1+\tau_k)\big)^{p_k^+-1},
\big(\lstar_k(0)(1+\tau_k)\big)^{p_k^--1}\Big\} {\mathcal
H}^{N-1}\big(\{u_k>0\}\cap Z_0\big).
\end{aligned}
\end{equation}

{}From now on, in order to simplify the computations, we assume that
$\lstar_k(0)\ge1$. The final result will be the same if not.

By \eqref{eq-1}, \eqref{eq-2} and \eqref{eq-3},
\[
\begin{aligned}
&
{\lstar_k(0)}^{p_k^--1}(1-C^{**}\rho_k^{\alpha^*})^{p_k^+-1}{\mathcal H}^{N-1}\big(\partial_{\rm{red}}\{u_k>0\}\cap Z_+\big)\\
&\ \ \ \le L_2\rho_k\big|\{u_k>0\}\cap Z_+\big|+
{\lstar_k(0)}^{p_k^+-1}(1+\tau_k)^{p_k^+-1}{\mathcal
H}^{N-1}\big(\{u_k>0\}\cap Z_0\big).
\end{aligned}
\]

Therefore,
\begin{equation}\label{eq-3a}
\begin{aligned}
&{\mathcal H}^{N-1}\big(\partial_{\rm{red}}\{u_k>0\}\cap Z_+\big)\\
&\qquad \le
{\lstar_k(0)}^{p_k^+-p_k^-}\Big(\frac{1+\tau_k}{1-C^{**}\rho_k^{\alpha^*}}\Big)^{p_k^+-1}
{\mathcal H}^{N-1}\big(\{u_k>0\}\cap Z_0\big)\\
&\qquad+\frac{L_2\rho_k}{{\lstar_k(0)}^{p_k^--1}(1-C^{**}\rho_k^{\alpha^*})^{p_k^+-1}}
\big|\{u_k>0\}\cap Z_+\big|.
\end{aligned}
\end{equation}

Now, we use the excess area formula Lemma 7.5 in \cite{AC} (with
$E_k=\{u_k>0\}\cup Z_-$) that states that, since $F(y_0)>g(y_0)$,
 \begin{equation}\label{eq-3b}
 {\mathcal H}^{N-1}\big(\partial_{\rm{red}}E_k\cap Z\big)\ge {\mathcal H}^{N-1}(Z_0)+c\sigma_k^2
 \end{equation}
for $k$ large. 

  Therefore, since there holds $Z\cap\partial E_k= \big(Z_+\cap\partial\{u_k>0\}\big)\cup \big(Z_0\cap\{u_k=0\}\big)$ and \eqref{eq-3b}, 
	we obtain
\begin{equation}\label{eq-3c}
\begin{aligned}
{\mathcal
H}^{N-1}\big(Z_+\cap\partial_{\rm{red}}\{u_k>0\}\big)&\ge
{\mathcal H}^{N-1}\big(Z\cap\partial_{\rm{red}}E_k\big)-{\mathcal H}^{N-1}\big(Z_0\cap\{u_k=0\}\big)\\
&\ge{\mathcal H}^{N-1}\big(Z_0\big)+c\sigma_k^2-{\mathcal H}^{N-1}\big(Z_0\cap\{u_k=0\}\big)\\
&={\mathcal H}^{N-1}\big(Z_0\cap\{u_k>0\}\big)+c\sigma_k^2.
\end{aligned}
\end{equation}

From here, using the   facts that
 \[
 {\lstar_k(0)}^{p_k^+-p_k^-}\Big(\frac{1+\tau_k}{1-C^{**}\rho_k^{\alpha^*}}\Big)^{p_k^+-1}
-1\le C_0\big(\tau_k+\rho_k^{\alpha^*}\big)
\]
and 
\[
\frac{L_2\rho_k}{{\lstar_k(0)}^{p_k^--1}(1-C^{**}\rho_k^{\alpha^*})^{p_k^+-1}}\le
C_1\rho_k,
\]
together with $|\{u_k>0\}\cap Z_+|\le |B_1|\le C$,  ${\mathcal H}^{N-1}(\{u_k>0\}\cap Z_0)\le{\mathcal H}^{N-1}(Z_0)\le C$, \eqref{eq-3a} and \eqref{eq-3c},   we get
\[
c\sigma_k^2\le CC_0(\tau_k+\rho_k^{\alpha^*})+CC_1\rho_k\le
C_2(\tau_k+\rho_k^{\alpha^*}).
\]

This is a contradiction to our assumptions that
$C^*\rho_k^{\alpha^*}\le \lstar_k(0) \tau_k$ and
$\frac{\tau_k}{\sigma_k^2}\to0$.
\end{proof}

The following lemma was proved in \cite{ACF} with $c=1$. The
result is obtained by rescaling the $h$ variable.
\begin{lemm}\label{lema 5} Let $w(y,h)$ be such that
\begin{enumerate}
\item[(a)] $\sum_{i=1}^{N-1}w_{y_iy_i}+c\,w_{hh}=0$ in
$B_1\cap\{h<0\}$ with $c>0$.

\item[(b)] $w(y,h)\to g$ in $L^1$ as $h\nearrow 0$.

\item[(c)] $g$ is subharmonic and continuous in $B_1'$, $g(0)=0$.

\item[(d)] $w(0,h)\le C|h|$.

\item[(e)] $w\ge -C$.
\end{enumerate}

Then, there exists $C_0$ depending only on $C$, N and  $c$ such that,
for every $y\in B_{1/2}'$,
\[
\int_0^{1/2}\frac1{r^2}\Big(\fint_{\partial
B_r'(y)}\,g(z)d{\mathcal H}^{N-2}\Big)\,dr\le C_0.
\]
\end{lemm}

Then, we have
\begin{lemm}\label{lema 6} Let $u_k, p_k, \lstar_k, f_k, \rho_k, \sigma_k$ as in Lemma \ref{lema 4}. Let $F_k^+, F_k^-$ and $F$ as in that lemma.
There exists $C=C(N, p_{\min},p_{\max}, \lone, \ltwo)$ such that, if $y_0\in B_{1/2}'$,
\begin{equation}\label{F}
\int_0^{1/4}\frac1{r^2}\Big(\fint_{\partial
B_r'(y_0)}\big(F-F(y_0)\big)\,d{\mathcal H}^{N-2}\Big)\,dr\le C.
\end{equation}
\end{lemm}
\begin{proof}
The proof follows the lines of the previously cited papers.
The idea is that the function $2\big(F(y_0+\frac12 y)-F(y_0)\big)$ will take the place of the function $g$ in Lemma \ref{lema 5}.

We
write down the proof for the reader's convenience since we cannot
assume that $\lstar_k(0)=1$ and we have a right hand side in the
equation that was not present in the previous papers. We let
$y_0\in B_{1/2}'$ and consider the functions $\bar
u_k(y,h)=2u_k(y_0+\frac12y, \sigma_kF_k^+(y_0)+\frac12h)$ in
$B_1$. From the fact that $u_k\in F(\sigma_k,\sigma_k;\tau_k)$ in $B_1$ we deduce that $\bar u_k\in
F(4\sigma_k,4\sigma_k;\tau_k)$ in  $B_1$.

In fact, we denote $(x',x_N)=(y_0+\frac12y, \sigma_kF_k^+(y_0)+\frac12h)$ and recall that $|F_k^+|\le 1$ . Then we have for $y\in B_1'$, $h>4\sigma_k$ that
$x_N>\sigma_kF_k^+(y_0)+2\sigma_k\ge \sigma_k$ implying that $\bar u_k(y,h)=0$.

On the other hand, for $y\in B_1'$, $h<-4\sigma_k$  we have 
$x_N<\sigma_kF_k^+(y_0)-2\sigma_k\le -\sigma_k$. This implies that $\bar u_k(y,h)=2u_k(x',x_N)\ge -2\lstar_k(0)[x_N+\sigma_k]\ge 
-\lstar_k(0)[h+4\sigma_k]$.

Finally, we see that $|\nabla \bar u_k(y,h)|=|\nabla  u_k(y_0+\frac12y, \sigma_kF_k^+(y_0)+\frac12h)|\le \lstar_k(0)(1+\tau_k)$ and we 
conclude that $\bar u_k\in F(4\sigma_k,4\sigma_k;\tau_k)$ in  $B_1$.

\medskip

Observe that by this change of variables the  function $F_k^+(y)$ has been replaced by $2\big(F_k^+(y_0+\frac12y)-F_k^+(y_0)\big)$.

\medskip

Thus, from now on we may assume that $u_k\in
F(4\sigma_k,4\sigma_k;\tau_k)$ in $B_1$ and $y_0=0$. Let
\[
w_k(y,h)=\frac{u_k(y,h)+\lstar_k(0) h}{\sigma_k}.
\]
Then, given $0<\delta<\frac12$, we take $k\ge k_\delta$ so that
$\lstar_k(0)/2\le |\nabla u_k|\le 2\lstar_k(0)$ in
$B_{1-\delta}\cap\{h\le -C_\delta\sigma_k\}$ with $C_\delta$ the
constant in Lemma \ref{lema 3} with $R=1-\delta$. We have
\begin{equation}\label{Tk}
{\mathcal
T}_kw_k:=\sum_{ij}b_{ij}^k(x){w_k}_{x_ix_j}+\sum_jb_j^k(x){w_k}_{x_j}= \frac{b_N^k}{\sigma_k}\lstar_k(0)+
\frac{f_k}{\sigma_k}|\nabla u_k|^{2-p_k}\quad\mbox{in}\quad
B_{1-\delta}\cap\{h\le -C_\delta\sigma_k\}.
\end{equation}
Here
$b^k_{ij}(x)=\delta_{ij}+(p_k(x)-2)\frac{{u_k}_{x_i}{u_k}_{x_j}}{|\nabla
u_k|^2}$ and $b_j^k(x)={p_k}_{x_j}\log|\nabla u_k|$. Therefore,
${\mathcal T}_k$ is a uniformly elliptic operator with ellipticity
and bounds of the coefficients independent of $k$. Namely, they
satisfy \eqref{bij} and
\[
|b^k|\le \bar C_0\rho_k
\]
(see \eqref{bj}).

On the other hand, the right hand side satisfies
\begin{equation}\label{cota-rhs}
\frac{b_N^k}{\sigma_k}\lstar_k(0)+\frac{f_k}{\sigma_k}|\nabla u_k|^{2-p_k}\le
K_0\frac{\rho_k}{\sigma_k}\to0\quad\mbox{as}\quad
k\to\infty.
\end{equation}

We will divide the proof into several steps.

(i) We prove that there exists a constant $C>0$ such that $\|w_k\|_{L^\infty(B^-_1)}\le C$.

In fact, recall that $u_k\in F(4\sigma_k,4\sigma_k;\tau_k)$ in $B_1$ so $u_k(0,0)=0$ and $|\nabla u_k|\le \lstar_k(0)(1+\tau_k)$. On the other hand, there holds that $u_k(y,h)=0$ if $h\ge 4\sigma_k$. Therefore,
 \[
 u_k(y,h)\le \lstar_k(0)(1+\tau_k)(4\sigma_k-h)
 \]
 so that, if $-K\le h\le 0$,
 \[
 w_k(y,h)\le 4\lstar_k(0)(1+\tau_k)-\lstar_k(0)\frac{\tau_k}{\sigma_k}h\le C.
 \]

On the other hand, if $h<-4\sigma_k$, since $u_k\in F(4\sigma_k,4\sigma_k;\tau_k)$ in $B_1$, by \eqref{item-flat-2} in Definition  \ref{def-flat},
\[
w_k(y,h)=\frac{u_k(y,h)+\lstar_k(0) h}{\sigma_k}\ge-\frac{\lstar_k(0)(h+4\sigma_k)-\lstar_k(0)h}{\sigma_k}=-4\lstar_k(0).
\]

Finally, if $-4\sigma_k\le h\le 0$,
\[\begin{aligned}
w_k(y,h)&\ge - \frac{\lstar_k(0)(1+\tau_k)(4\sigma_k-h)-\lstar_k(0)h}{\sigma_k}\\
&=-4\lstar_k(0)(1+\tau_k)+\frac{\lstar_k(0)(2+\tau_k)h}{\sigma_k}\\
&\ge -C.
\end{aligned}
\]

\medskip

(ii) Uniform bounds of first and second order derivatives.

Recall that $w_k$ satisfies \eqref{Tk} that is uniformly elliptic with ellipticity constants and bounds of the coefficients independent of $k$ in $B_{1-\delta}\cap\{h<-C_\delta\sigma_k\}$. By step (i) we then have
\begin{equation}\label{C1alpha}
\|\ w_k\|_{{C^{1,\alpha}}({\mathcal K})}\le C_{\mathcal K}\quad\forall\ \
{\mathcal K}\subset\subset B_1^-.
\end{equation}
and, for every $1<q<\infty$,
\begin{equation}\label{W2q}
\|\ w_k\|_{{W^{2,q}}({\mathcal K})}\le C_{\mathcal K}\quad\forall\ \
{\mathcal K}\subset\subset B_1^-.
\end{equation}

Hence, for a subsequence that we still call $w_k$, there exists $w\in C^{1,\alpha}\cap W^{2,q}$ such that $w_k\to w$ in $C^1({\mathcal K})$ and weakly in $W^{2,q}({\mathcal K})$ for every ${\mathcal K}\subset\subset B_1^-.$

\medskip

(iii) Determining the equation satisfied by $w$.

\medskip

Let $c_{ij}=\delta_{ij}+(p_0-2)\delta_{iN}\delta_{jN}$ where $p_{\min}\le
p_0\le p_{\max}$ is the uniform limit of the sequence of functions
$p_k$ (for a subsequence). Then,  $b^k_{ij}\to c_{ij}$ uniformly
on compact subsets of $B_1^-$. In fact, by the uniform estimates
of the gradient of $w_k$ we have that
\begin{equation}\label{lim-grad}
\big|\nabla u_k(y,h)+\lambda_k^*(0)e_N\big|=\big|\nabla\big(u_k(y,h)+\lambda_k^*(0)h\big)\big|\le C_{\mathcal K}\sigma_k
\end{equation}
if $k\ge k_{\mathcal K}$ and $\mathcal K\subset\subset B_1^-$.

Let $\lambda_0^*=\lim_{k\to\infty}\lambda_k^*(0)$ (for a
subsequence). Then, by \eqref{lim-grad} $\nabla u_k\to -\lambda_0^*e_N$ uniformly on
compact subsets of $B_1^-$. Since
$\lambda_0^*\ge\lambda_{\min}>0$, there holds that
\[
\frac{{u_k}_{x_i}{u_k}_{x_j}}{|\nabla u_k|^2}\to \delta_{iN}\delta_{jN}
\]
uniformly on compact subsets of $B_1^-$. And we have proved the convergence.

On the other hand, $|b^k_j(x)|\le C_0 \sigma_k$. Therefore, by passing to the limit in \eqref{Tk} we get
\begin{equation}\label{eq-w}
\sum_{ij}c_{ij}w_{x_ix_j}=0\quad\mbox{in}\quad B_1^-.
\end{equation}

\medskip

(iv) Bounds of $w$.

Recalling that $|\nabla u_k|\le \lambda_k^*(0)(1+\tau_k)$, we get
\begin{equation}\label{bound-partial-wk}
 \frac\partial{\partial h}w_k(y,h)\ge-\frac{\lstar_k(0)(1+\tau_k)-\lstar_k(0)}{\sigma_k}=-\lstar_k(0)\frac{\tau_k}{\sigma_k}.
 \end{equation}
 Thus, for $h<0$,
 \begin{equation}\label{bound-negative}
  w_k(0,h)\le \lstar_k(0)\frac{\tau_k}{\sigma_k}|h|\to0\quad\mbox{as}\quad k\to\infty.
\end{equation}

Passing to the limit, we find that
\[
w(0,h)\le 0\mbox\quad{for}\quad h<0.
\]

\medskip

(v) Let us see that $w(y,h)
\to \lambda_0^*F(y)$ as $h\to0^-$, uniformly in $B'_{1-\delta}$ for every $0<\delta<1$.

First, as in \cite{ACF,DP1}, we can prove that
\begin{equation}\label{limitwk}
w_k(y,\sigma_k h)-\lstar_0F(y)\to 0\quad\mbox{uniformly in}\quad
B'_{1-\delta}\times[-K,-2C_\delta]
\end{equation}
for every $K>2C_\delta$ and every $0<\delta<1$. We omit this
proof, that relies heavily on Proposition \ref{prop-flatness} (see
\cite{ACF} for the proof).

In order to get the result,  following the ideas in \cite{ACF,DP1}, we construct a
barrier. First, for $\delta>0$ we let $\Omega_\delta$ a smooth
domain such that
\[
B_{1-2\delta}^-\subset\Omega_\delta\subset B_{1-\delta}^-.
\]
For $\ep>0$ small, we let $g_\ep\in C^3(\partial\Omega_\delta)$
such that $\|g_\ep\|_{C^3(\partial\Omega_\delta)}\le C$ with $C$
independent of $\ep$ and $\delta$ and
\[
\begin{matrix}
\lstar_0F-2\ep\le&g_\ep\le \lstar_0F-\ep &\qquad\qquad\mbox{in }\partial\Omega_\delta\cap\partial B_{1-3\delta}^-\cap\{h=0\}\\
&g_\ep\le \lstar_0F-\ep &\mbox{in }\partial\Omega_\delta\cap\{h=0\}\\
&\hskip-12pt g_\ep\le w-\ep &\mbox{in
}\partial\Omega_\delta\cap\{h<0\}.
\end{matrix}
\]
Then, we let $\phi_\ep$ the solution to
\[
\begin{cases}
\sum_{ij}c_{ij}{\phi_\ep}_{x_ix_j}=1\quad&\mbox{in}\quad \Omega_\delta\\
\phi_\ep= g_\ep\quad&\mbox{on}\quad\partial\Omega_\delta
\end{cases}
\]
with $c_{ij}$ as in \eqref{eq-w}.

On one hand, if $k\ge k(\ep,\delta)$,
\[
\phi_\ep\le
w_k\quad\mbox{on}\quad\partial\Omega_\delta\cap\{h<-2C_\delta\sigma_k\}.
\]

On the other hand, since
$\|\phi_\ep\|_{C^2(\overline{\Omega_\delta})}\le C$, there holds
that, for  $K>2C_\delta$  and $k\ge k(\ep,\delta,K)$,
\[
\phi_\ep\le w_k\quad\mbox{on}\quad\Omega_\delta\cap\{h=
-K\sigma_k\}.
\]

Recall that, by Lemma \ref{lema 3}, we have  
\[
|\nabla u_k|\ge \frac{\lstar_k(0)}2\quad\mbox{in}\quad
B_{1-\delta}\cap\{h<-C_\delta\sigma_k\}
\]
and there holds \eqref{Tk} and \eqref{cota-rhs}. Therefore,
\[
{\mathcal T}_kw_k\le K_0\frac{\rho_k}{\sigma_k}\le
\frac12\quad\mbox{in}\quad\Omega_\delta\cap\{h<-K\sigma_k\}
\]
if $k\ge k_0$.

Let us see that
\begin{equation}\label{phi-ep}
{\mathcal T}_k\phi_\ep\ge
\frac12\quad\mbox{in}\quad\Omega_\delta\cap\{h<-K\sigma_k\}
\end{equation}
if $K$ is large independently of $\ep$ and $k$ is large
independently of $\ep$ and $K$. In fact, for $x\in\Omega_\delta$,
\[
\begin{aligned}
{\mathcal
T}_k\phi_\ep&=\sum_{ij}c_{ij}{\phi_\ep}_{x_ix_j}+\sum_{ij}\big(b^k_{ij}(x)-c_{ij}\big)
{\phi_\ep}_{x_ix_j}+\sum_jb^k_j(x){\phi_\ep}_{x_j}\\
&\ge
1-\|D^2\phi_\ep\|_{L^\infty}\sum_{ij}\|b_{ij}^k-c_{ij}\|_{L^\infty}-\|b^k\|_{L^\infty}\|\nabla\phi_\ep\|_{L^\infty}.
\end{aligned}
\]

On one hand, $\|b^k\|_{L^\infty}\le C_0\sigma_k\to0$ as
$k\to\infty$. On the other hand, by elliptic estimates up to the boundary $\{h=-K\sigma_k\}$, since we have proved that $|w_k|\le C$,
\[
\begin{aligned}
\|\nabla &(u_k+\lstar_k(0)h)\|_{L^\infty(\{h\le -K\sigma_k\})}={\sigma_k}\|\nabla w_k\|_{L^\infty(\{h\le -K\sigma_k\})}\\
&\le \sigma_k
C\frac{\rho_k/\sigma_k+1}{(K-C_\delta)\sigma_k}\le \frac
{2C}{K-C_\delta}\quad\mbox{in}\quad\Omega_\delta\cap\{h<-K\sigma_k\}.
\end{aligned}
\]

 Then, as $\frac{\lstar_k(0)}2\le
|\nabla u_k|\le 2\lstar_k(0)$  in that set and $p_k(x)-p_0\to0$ uniformly in $B_1$,
\[
\|b_{ij}^k-c_{ij}\|_{L^\infty(B_{1}\cap\{h\le-K\sigma_k\})}\le
\frac C{K-C_\delta}+o_k(1).
\]

We conclude, by taking $K$ large enough independent of $k$ and
$\ep$ and then, $k$ large,  that \eqref{phi-ep} holds.

Therefore, $\phi_\ep\le w_k$ in
$\Omega_\delta\cap\{h\le-K\sigma_k\}$. By letting $k\to\infty$ we
find that $\phi_\ep\le w$ in $\Omega_\delta\cap\{h<0\}$ and then,
by letting $h\to0^-$,
\[
\liminf_{h\to0^-}w(y,h)\ge\lim_{h\to0^-}\phi_\ep(y,h)\ge\lstar_0F(y)-2\ep\quad\mbox{for}\quad
y\in B'_{1-3\delta}.
\]

In order to get a bound from above, we recall
\eqref{bound-partial-wk} and get,
\[
w_k(y,h)-w_k(y,-K\sigma_k)\le
-C\frac{\tau_k}{\sigma_k}|h|\quad\mbox{if}\quad h\le-K\sigma_k.
\]

On the other hand, $w_k(y,-K\sigma_k)\to \lstar_0F(y)$ uniformly
in $B'_{1-\delta}$. Hence, if $k$ is large, and $(y,h)\in
B_{1-\delta}^-\cap\{h\le -K\sigma_k\}$,
\[
w_k(y,h)\le \lstar_0F(y)+2\ep
\]
and we deduce that, for $(y,h)\in B_{1-\delta}^-$,
\[
w(y,h)\le \lstar_0F(y)+2\ep.
\]
Therefore,
\[
\limsup_{h\to0^-}w(y,h)\le \lstar_0F(y)+2\ep\quad\mbox{uniformly
in}\quad B_{1-\delta}'.
\]

Since $\ep$ is arbitrary, we conclude that, for every $0<\delta<1$,
\[
\lim_{h\to0^-}w(y,h)=\lstar_0F(y)\quad\mbox{uniformly for }y\in
B_{1-3\delta}'.\]

\medskip

(vi) Final step.

We apply Lemma \ref{lema 5} to the function $w$ and recall that when writing $w(y,0)$ in the original variables we get
$2\big(F(y_0+\frac12y)-F(y_0)\big)$. So, the result is proved.
\end{proof}

\begin{coro} \label{coro-ell} Let $u_k, p_k, \lstar_k, f_k, \rho_k, \sigma_k$ and $F$ as in Lemma \ref{lema 4}.
There exists a constant $C=C(N, p_{\min}, p_{\max}, \lone, \ltwo)$ and, for every $0<\theta<1$ there exist $c_\theta=c_\theta(N, p_{\min}, p_{\max}, \lone, \ltwo, \theta)$, a ball $B'_r$ and $\ell\in \R^{N-1}$ such that
\[
c_\theta\le r\le \theta,\quad|\ell|\le C,\quad F(y)\le \ell\cdot
y+\frac\theta2r\quad\mbox{for }|y|\le r.
\]
\end{coro}
\begin{proof}
The result is a consequence of Lemma \ref{lema 6} and the proof
follows as Lemmas 7.7 and 7.8 in \cite{AC}.
\end{proof}

Now, we apply the corollary to a weak flat solution $u$ if
$\sigma$ is small enough.
\begin{lemm}\label{lema 8} Let $p\in Lip(B_\rho)$, $\lstar\in C^{\alpha^*}(B_\rho)$, $f\in L^\infty(B_\rho)$ such that
$1<p_{\min}\le p(x)\le p_{\max}<\infty$, $0<\lone\le\lstar(x)\le \ltwo<\infty$ with $|\nabla p|\le L_1$, $|f|\le L_2$ and
$[\lstar]_{C^{\alpha^*}(B_\rho)}\le C^*$. Let $0<\theta<1$.
There exists $\sigma_\theta=\sigma_\theta(\theta, N, p_{\min}, p_{\max}, \lone, \ltwo, L_1, L_2, C^*)$ such that,
if $$u\in F(\sigma,\sigma;\tau)\mbox{ in }B_\rho \mbox{ in direction }\nu$$ with power $p$, slope $\lstar$ and rhs $f$ and, if $C^*\rho^{\alpha^*}\le \lstar(0)\tau$, $\sigma\le\sigma_\theta$ and $\tau\le \sigma_\theta \sigma^2$ there holds that
$$u\in F(\theta\sigma,1;\tau) \mbox{ in }B_{\bar\rho}\mbox{ in direction }\bar\nu$$
with the same power, slope and rhs and
\[
c_\theta\rho\le\bar\rho\le\theta\rho,\qquad |\nu-\bar\nu|\le
C\sigma.
\]

Here $c_\theta$ and $C$ are the constants in Corollary
\ref{coro-ell}.
\end{lemm}
\begin{proof} It follows as Lemma 7.9 in \cite{AC} by applying Corollary \ref{coro-ell} to $\bar u_k(x)=\frac1{\rho}_k u_k({\rho}_k x)$.
\end{proof}

Now, in order to improve on the gradient in the flatness class, we
find an equation to which $v=|\nabla u|$ is a subsolution.
\begin{lemm}\label{lema 9} Let
$p\in W^{1,\infty}(\Omega)\cap W^{2,q}(\Omega)$ with
$1<p_{\min}\le p(x)\le p_{\max}<\infty$ in $\Omega$ and $f\in
L^\infty(\Omega)\cap W^{1,q}(\Omega)$ for some $q\ge1$.

 Let $u$ such that $\Delta_{p(x)}u=f$ and $0<c\le|\nabla u|\le C$ in $\Omega$. There exist $D=\{D_{ij}\}$, $ B=\{ b_j\}$ and $G$ such that
\[\begin{aligned}
&\bar\beta|\xi|^2\le \sum_{ij}D_{ij}(x)\xi_i\xi_j\le{\bar\beta}^{-1}|\xi|^2\quad\mbox{for every }\xi\in\R^N,\ x\in\Omega,\\
&\|B\|_{L^{\infty}(\Omega)} \le \bar
C\quad,\quad\|G\|_{L^q(\Omega)}\le \bar C
\end{aligned}
\]
with $\bar\beta=\bar\beta(p_{\min}, p_{\max}, c, C)>0$, $\bar C=\bar
C(p_{\min}, p_{\max}, c, C,\|f\|_{L^\infty(\Omega)\cap W^{1,q}(\Omega)}, \|p\|_{W^{1,\infty}(\Omega)\cap W^{2,q}(\Omega)})$ such that $v=|\nabla u|$ satisfies
\begin{equation}\label{eq-v1}
\mbox{div\,}D\nabla v+ B\cdot\nabla v\ge G
\end{equation}
weakly in $\Omega$.
\end{lemm}
\begin{proof}
We start with some notation.
For $x\in \Omega$, $\xi\in\R^N$, we let
$A(x,\xi)=|\xi|^{p(x)-2}\xi$.
First we observe that, by the arguments in Theorem 3.2 in \cite{CL2}, $u\in W^{2,2}_{\rm loc}(\Omega)$ and then, by using the nondivergence
form of the equation, we deduce that $u\in W^{2,t}_{\rm loc}(\Omega)$ for every $1\le t<\infty$ (see Lemma 9.16 in \cite{GT}).

Then,  taking $\eta\in
C_0^\infty(\Omega)$, letting $\eta_{x_k}$ as test function and
integrating by parts, we get
\begin{equation}\label{weak}
\int f\eta_{x_k}=\int\frac{\partial A}{\partial x_k}(x,\nabla
u)\nabla\eta+\sum_{ij}\int a_{ij}(x,\nabla u)u_{x_jx_k}\eta_{x_i}
\end{equation}
where $a_{ij}(x,\xi)=\frac{\partial A_i}{\partial \xi_j}(x,\xi)$.

Observe that \eqref{weak} actually holds for any $\eta\in
W^{1,p(x)}_0(\Omega)$.

Then, we  take $\eta=u_{x_k}\psi$ with $0\le\psi\in
C_0^\infty(\Omega)$ arbitrary.  Hence, by using the ellipticity of
$a_{ij}$ and after summation on $k$, we get
\[\begin{aligned}
\int f\Delta u \psi+\int f\langle\nabla u,\nabla \psi\rangle&\ge\sum_{i,k}\int\frac{\partial A_i}{\partial x_k}(x,\nabla u) u_{x_ix_k}\psi\\
&+\sum_{i,k}\int \frac{\partial A_i}{\partial x_k}(x,\nabla u) u_{x_k}\psi_{x_i}+\sum_{i,j}\int a_{ij}\sum_k u_{x_k}u_{x_jx_k}\psi_{x_i}.
\end{aligned}
\]
Now, we denote $D=(D_{ij})$ with $D_{ij}=|\nabla u|a_{ij}$, we use that
$v_{x_j}=\sum_k\frac{u_{x_kx_j}u_{x_k}}{|\nabla u|}$ and we integrate by parts the second terms on the left and right hand sides.
In fact, since
\[\frac{\partial A_i}{\partial x_k}(x,\nabla u)=|\nabla u|^{p(x)-2}\log|\nabla u| u_{x_i}
p_{x_k},
\]
we get
\begin{equation}\label{choclo}\begin{aligned}
\frac{ d}{dx_i}\Big[\frac{\partial A_i}{\partial x_k}(x,\nabla u)\Big] &=
|\nabla u|^{p(x)-2}\big(\log|\nabla u|\big)^2 \,u_{x_i}p_{x_k}p_{x_i}\\&+
|\nabla u|^{p(x)-2}\log|\nabla u|\, u_{x_i} p_{x_kx_i}+|\nabla u|^{p(x)-2}\log|\nabla u|\, u_{x_ix_i}p_{x_k}\\
&+(p(x)-2)|\nabla u|^{p(x)-3}\log|\nabla u| u_{x_i}p_{x_k}v_{x_i}
+
|\nabla u|^{p(x)-3} u_{x_i}p_{x_k}v_{x_i},
\end{aligned}\end{equation}
so we obtain
\begin{equation}\label{eq-v2}\begin{aligned}
-\int\langle\nabla f,\nabla u\rangle\psi&\ge \int\langle D\nabla v,\nabla \psi\rangle+\sum_{i,k}\int \frac{\partial A_i}{\partial x_k}(x,\nabla u) u_{x_ix_k}\psi\\
&-\sum_{i,k}\int \frac{ d}{dx_i}\Big[\frac{\partial A_i}{\partial x_k}(x,\nabla u)\Big] u_{x_k}\psi-\sum_{i,k}\int \frac{\partial A_i}{\partial x_k}(x,\nabla u) u_{x_ix_k}\psi\\
&=\int\langle D\nabla v,\nabla \psi\rangle
-\sum_{i,k}\int \frac{ d}{dx_i}\Big[\frac{\partial A_i}{\partial x_k}(x,\nabla u)\Big] u_{x_k}\psi.
\end{aligned}\end{equation}

Then, by replacing \eqref{choclo} in \eqref{eq-v2}, it follows
\[\begin{aligned}
-\int\langle\nabla f,\nabla u\rangle\psi&\ge \int\langle D\nabla v,\nabla \psi\rangle
-\int |\nabla u|^{p(x)-2}\big(\log|\nabla u|\big)^2 \langle\nabla u,\nabla p\rangle^2\psi\\
&- \int |\nabla u|^{p(x)-2}\log|\nabla u|\sum_{i,k} u_{x_i}u_{x_k} p_{x_kx_i}\psi
-\int |\nabla u|^{p(x)-2}\log|\nabla u| \langle \nabla u,\nabla p\rangle\Delta u\psi\\
&-\int\Big\langle |\nabla u|^{p(x)-3}\big[(p(x)-2)\log|\nabla u|+1\big]\langle\nabla u,\nabla p\rangle\ \nabla u,\nabla v\Big\rangle \psi.
\end{aligned}\]

Finally, since $|\nabla u|^{p(x)-2}\Big(\Delta u+(p(x)-2)\sum_{i,j}\frac{u_{x_i}u_{x_j}}{|\nabla u|^2}u_{x_ix_j}+\log|\nabla u| \langle \nabla u,\nabla p\rangle\Big)=f$,
\[\begin{aligned}
-\int |\nabla u|^{p(x)-2}\log|\nabla u| \langle \nabla u,\nabla p\rangle\Delta u\psi&=-
\int f \log|\nabla u| \langle \nabla u,\nabla p\rangle\psi\\&+
\int \big\langle(p(x)-2)|\nabla u|^{p(x)-3}\log|\nabla u| \langle \nabla u,\nabla p\rangle\nabla u,\nabla v\Big\rangle \psi\\&+
\int |\nabla u|^{p(x)-2}\big(\log|\nabla u|\big)^2 \langle\nabla u,\nabla p\rangle^2\psi.
\end{aligned}
\]

Hence, $v$ satisfies \eqref{eq-v1} with
\[\begin{aligned}
D_{ij}&=|\nabla u|^{p(x)-1}\big(\delta_{ij}+\frac{(p(x)-2)}{|\nabla u|^2} u_{x_i}
u_{x_j}\big),\\
 B&=|\nabla u|^{p(x)-3}\langle\nabla u,\nabla p\rangle\  \nabla u,\\
 G&=\langle \nabla f,\nabla u\rangle-f \log|\nabla u| \langle \nabla u,\nabla p\rangle -|\nabla u|^{p(x)-2}\log|\nabla u|\sum_{i,k} u_{x_i}u_{x_k} p_{x_kx_i}.
 \end{aligned}\]
\end{proof}

\begin{rema}\label{por-ref-lyagh} A similar lemma to Lemma \ref{lema 9}, valid for the case $f\equiv 0$, was established in reference
\cite{CL1} (Lemma 2.2).
\end{rema}

Now, we get an estimate on $|\nabla u|$ close to the free boundary.

\begin{lemm}\label{Thm4.2} Let $p$ and $f$ as in Lemma
\ref{lema 9} with $q>\max\{1,N/2\}$ and $\lstar\in
C^{\alpha^*}(\Omega)$ with $0<\lambda_{\min}\le \lstar(x)\le
\lambda_{\max}<\infty$ in $\Omega$ and
$[\lstar]_{C^{\alpha^*}(\Omega)}\le C^*$. Let $u$ be a weak
solution to \ref{fbp-px} in $\Omega$ and let
$x_0\in\Omega\cap\partial\{u>0\}$ with $B_{4R}(x_0)\subset\Omega$,
$R\le 1$. Assume that, for every $r\le R$,
\[
u\in F(\sigma,1;\infty)\quad\mbox{in}\quad B_r(x_0)\quad\mbox{in some direction}\quad\nu_r,
\]
with power $p$, slope $\lstar$ and rhs $f$, with $\sigma\le 1/2$.

Then, for every $x_1$ in $B_r(x_0)$,
\begin{equation}
\label{bound-grad} |\nabla u|\le \lstar(x_1)+C
\left(\frac{r}{R}\right)^\gamma\quad\mbox{in}\quad
B_r(x_1)\quad\mbox{if}\quad r\le R,
\end{equation}
for some constants $C$ and $0<\gamma<1$ depending only on $N$,
$p_{\min}$, $p_{\max}$, $\lambda_{\min}$,
$\|f\|_{L^\infty(B_{2R}(x_0))\cap W^{1,q}(B_{2R}(x_0))}$,
$\|p\|_{W^{1,\infty}(B_{2R}(x_0))\cap W^{2,q}(B_{2R}(x_0))}$,
$\alpha^*$, $C^*$, $q$ and $\|\nabla u\|_{L^\infty(B_{2R}(x_0))}$.

\end{lemm}
\begin{proof} We let $0<R_0\le R$,  $\ep>0$ and define
\[\begin{aligned}
\lstar_{2R_0}&=\sup_{B_{2R_0}(x_0)}\lstar(x),\\
 U_\ep(x)&=\big(|\nabla u|-\lstar_{2R_0}-\ep\big)^+.
 \end{aligned}\]
Let $0<r\le R_0$. Since for every $ \bar x\in \overline{B_{2R_0}}(x_0)\cap\partial\{u>0\}$
\[
\limsup_{\stackrel{x\to \bar x}{u(x)>0}}|\nabla u|\le \lstar(\bar x),
\]
then the function $U_\ep$
vanishes in a neighborhood of $B_{2r}(x_0)\cap\partial\{u>0\}$.

We have $|\nabla u|\ge \lambda_{\min}$ in $\{U_\ep>0\}$ and
moreover, arguing as in Lemma \ref{lema 9} we see that $u\in
W^{2,t}(B_{2r}(x_0)\cap\{U_\ep>0\})$ for every $1\le t<\infty$.
Thus, by Lemma \ref{lema 9}, $U_\ep$ is a solution to
\begin{equation*}
\mbox{div} D\nabla U_\ep+ B\cdot\nabla U_\ep\ge G
\end{equation*}
in $\{U_\ep>0\}\cap B_{2r}(x_0)$
for some functions $D=\{D_{ij}\}$, $ B=\{ b_j\}$ and $G$ such that
\begin{equation}\label{cotas-eq}
\begin{aligned}
&\bar\beta|\xi|^2\le \sum_{ij}D_{ij}(x)\xi_i\xi_j\le{\bar\beta}^{-1}|\xi|^2\quad\mbox{for every }\xi\in\R^N,\ x\in B_{2R}(x_0),\\
&\|B\|_{L^{\infty}(\{U_\ep>0\}\cap B_{2R}(x_0))}\le \bar
C\quad,\quad\|G\|_{L^q(\{U_\ep>0\}\cap B_{2R}(x_0))}\le \bar C
\end{aligned}
\end{equation}
with $\bar\beta=\bar\beta(p_{\min}, p_{\max}, \lambda_{\min},
\|\nabla u\|_{L^\infty(B_{2R}(x_0))})$, $\bar C=\bar C(p_{\min},
p_{\max}, \lambda_{\min}, \|\nabla u\|_{L^\infty(B_{2R}(x_0))},
\linebreak \|f\|_{L^\infty(B_{2R}(x_0))\cap W^{1,q}(B_{2R}(x_0))},
 \|p\|_{W^{1,\infty}(B_{2R}(x_0))\cap W^{2,q}(B_{2R}(x_0))})$.

Therefore, if $\widetilde G$ and $\widetilde B$ are the extensions by 0 of $G$ and $B$ respectively from $\{U_\ep>0\}\cap B_{2r}(x_0)$ to $B_{2r}(x_0)$ and $\widetilde D$ is an extension of $D$ that preserves the uniform ellipticity with the same constants, there holds that $U_\ep$ satisfies
\begin{equation}\label{eq-Uep}
\mbox{div} \widetilde D\nabla U_\ep+ \widetilde B\cdot\nabla
U_\ep\ge \widetilde G
\end{equation}
in $B_{2r}(x_0)$  (see, for instance, Lemma 2.1 in \cite{LW3}).

Let now $h_\ep(r)=\sup_{B_{r}(x_0)}U_\ep$ and $V=h_\ep(2r)-U_\ep$. Then,
\begin{equation*}
\mbox{div} \widetilde D\nabla V+ \widetilde B\cdot\nabla V\le
-\widetilde G\quad\mbox{in}\quad B_{2r}(x_0).
\end{equation*}
Moreover, $V\ge0$ in $B_{2r}(x_0)$. By the weak Harnack inequality (see \cite{GT}),
\[
\inf_{B_{r}(x_0)} V+r^{2-N/q}\|\widetilde G\|_{L^q(B_{2r}(x_0))}\ge c\pint_{B_{3r/2}(x_0)}V
\]
with $c=c(N,\bar\beta,\|\widetilde B\|_{L^\infty(B_{2R}(x_0))}, q)$.

Now, since by the flatness condition, $u$ (and therefore $U_\ep$) vanishes in the ball $B_{\frac{1-\sigma}2 r}(x_0+\frac{1+\sigma}2 r\nu_r)$ for some direction $\nu_r$, there holds that   $V=h_\ep(2r)$ in $B_{\frac{1-\sigma}2 r}(x_0+\frac{1+\sigma}2 r\nu_r)$ and therefore,
\[
h_\ep(2r)-h_\ep(r)+r^{2-N/q}\bar C\ge \hat c \big(\frac{1-\sigma}2\big)^N h_\ep(2r)\ge
{\bar c}\, h_\ep(2r)
\]
since $\sigma\le 1/2$, with $\bar c=\bar c(N,\bar\beta,\|\widetilde B\|_{L^\infty(B_{2R}(x_0))},q)<1$ and $\bar C$ the constant in \eqref{cotas-eq}.
We pass to the limit as $\ep\to0$ and we conclude that
\begin{equation}\label{estim-para-h(r)}
h(r)\le \big(1-\bar c\big) h(2r)+r^{2-N/q}\bar C,
\end{equation}
if $r\le R_0$ with $h(r)=\sup_{B_r(x_0)}\big(|\nabla u|-\lstar_{2R_0}\big)^+$. Since $2-N/q>0$, there exist $\tilde\gamma\in(0,1)$, $\tilde C>0$ depending only on
$N,q,\bar c, \|\nabla u\|_{L^\infty(B_{2R}(x_0))}$ and $\bar C$ such that
\[h(s)\le \tilde C \big(\frac{s}{2R_0}\big)^{\tilde\gamma}\]
if $s\le 2R_0$. This implies
\begin{equation}\label{por-iterac}
\sup_{B_{2r}(x_0)}|\nabla u|\le \sup_{B_{2R_0(x_0)}}\lstar (x) + \tilde C\big(\frac{r}{R_0}\big)^{\tilde\gamma},
\end{equation}
if $r\le R_0\le R$, and the  H\"older continuity of $\lstar(x)$ gives, for $x_1\in B_{2R_0}(x_0)$,
\begin{equation}\label{por-hold}
\sup_{B_{2R_0(x_0)}}\lstar (x)\le  \lstar(x_1) + C^*(4R_0)^{{\alpha}^*}.
\end{equation}

We now take $r\le R$,  $R_0=r^{1/2}R^{1/2}$ and $x_1\in B_r(x_0)$ and obtain,  from \eqref{por-iterac} and \eqref{por-hold},
\[
\sup_{B_{r}(x_1)}|\nabla u|\le\sup_{B_{2r}(x_0)}|\nabla u|\le
\lstar(x_1) + C\left(\frac{r}{R}\right)^\gamma,
\]
for $\gamma=\min\{\frac{{\alpha}^*}{2}, \frac{\tilde\gamma}{2}\}$
and $C$ depending only on $\tilde C$, $C^*$, $\tilde\gamma$ and
${\alpha}^*$, which proves \eqref{bound-grad} and completes the
proof.
\end{proof}

Let us show that a point $x_0$ in the reduced free boundary of a weak solution is always under the assumptions of Lemma \ref{Thm4.2}.
\begin{lemm}\label{lema-reduced-Thm4.2}
Let $p\in Lip(\Omega)$ with $1<p_{\min}\le p(x)\le
p_{\max}<\infty$, $\lstar\in C(\Omega)$ with
$0<\lone\le\lstar(x)\le\ltwo<\infty$ and $f\in
L^\infty(\Omega)$.
Let $u$ be a weak solution to \ref{fbp-px} in $\Omega$ and $x_0\in\Omega\cap\partial_{\rm{red}}\{u>0\}$.

There exists $\sigma_0>0$   such that, if $\sigma<\sigma_0$, there exists $r_{\sigma}>0$ such that, for every $r\le r_{\sigma}$,
\[
u\in F(\sigma,1;\infty)\quad \mbox{ in }\quad B_r(x_0)\quad\mbox{ in direction }\quad \nu(x_0),
\]
with power $p$, slope $\lstar$ and rhs $f$. Here $\nu(x_0)$ denotes the exterior unit normal to $\Omega\cap\partial\{u>0\}$ at $x_0$ in the measure
theoretic sense.
\end{lemm}
\begin{proof} Assume for simplicity that $x_0=0$ and $\nu(x_0)=e_N$. Let $R>0$ be such that $B_{4R}\subset\Omega$.

Given $0<\varepsilon<\frac12$,
there exists $r_\varepsilon\le R$ such that
\begin{equation}\label{por-nor-med}
\frac{|\{u>0\}\cap B_r^+|}{|B_r|}<\varepsilon\quad \mbox{ if }\quad r\le r_{\varepsilon},
\end{equation}
and also a constant  $c_N>1$ so that
\begin{equation}\label{no-todo-posit}
|B_r^+\setminus\{0<x_N<\sigma r\}|\ge |B_r|(1/2-c_N \sigma)> \varepsilon |B_r| \quad\mbox{ if }\quad \sigma<\frac{1/2-\varepsilon}{c_N}.
\end{equation}
Let $r\le \frac{r_{\varepsilon}}{2}$ and suppose there exists $\bar x\in (B_r^+\setminus\{0<x_N<\sigma r\})\cap\partial\{u>0\}$.
Then, $\sup_{B_{\rho}(\bar x)}u\ge c_{\min}\rho$, if $\rho\le \rho_0=\min\{r_0, R\}$, with $c_{\min}$ and $r_0$ the
constants corresponding to $D=B_{2R}$ in the definition of weak solution.

Then, if $r\le \rho_0$, there exists $x_1\in \bar B_{\sigma r/2} (\bar x)$ such that $u(x_1)\ge c_{\min}\sigma r / 2$, implying that
\[
u(x)\ge c_{\min}\sigma r/2  - L\kappa \sigma r/2 >0 \quad \mbox{ in }\quad B_{\kappa\sigma r /2}(x_1)\subset B_{2r}^+,
\]
if $\kappa\le \min\{1, \frac{c_{\min}}{2L}\}$, where $L$ is the Lipschitz constant of $u$ in $B_{2R}$. As a consequence,
\[
\frac{|\{u>0\}\cap B_{2r}^+|}{|B_{2r}|}\ge (\kappa\sigma/4)^N,
\]
which contradicts \eqref{por-nor-med} if $(\kappa
\sigma/4)^N>\varepsilon$. Finally,  we fix $\sigma_0=(2
c_N)^{-1}$, take $\sigma<\sigma_0$ and choose $0<\varepsilon<
\frac12$ satisfying
\begin{equation*}
\frac{4}{\kappa} \varepsilon^{1/N}<\sigma<\frac{1/2-\varepsilon}{c_N}.
\end{equation*}
Then, letting  $r_{\sigma}=\min\{\frac{r_{\varepsilon}}{2},
\rho_0\}$ and $r\le r_{\sigma}$, we observe that
$(B_r^+\setminus\{0<x_N<\sigma r\})\cap\partial\{u>0\}=\emptyset$ by the above discussion, and that we cannot have $u>0$ in
$B_r^+\setminus\{0<x_N<\sigma r\}$ because of \eqref{por-nor-med} and \eqref{no-todo-posit}.
Therefore we conclude that $u\in F(\sigma,1;\infty)$ in $B_r$ with power $p$, slope $\lstar$ and rhs $f$, for
every $r\le r_{\sigma}$.
\end{proof}

Now,   we get a result that holds at  free boundary points satisfying a density condition on the zero set. This is the situation when
$u$ comes from a minimization problem as was the case in \cite{AC,ACF,DP1}, for instance.
\begin{lemm}\label{stronger-Thm4.2}
Let $p$ and $f$ as in Lemma \ref{lema 9} with $q>\max\{1,N/2\}$
and $\lstar\in C^{\alpha^*}(\Omega)$ with $0<\lambda_{\min}\le
\lstar(x)\le \lambda_{\max}<\infty$ in $\Omega$ and
$[\lstar]_{C^{\alpha^*}(\Omega)}\le C^*$. Let $u$ be a weak
solution to \ref{fbp-px} in $\Omega$ and let
$x_0\in\Omega\cap\partial\{u>0\}$ with $B_{4R}(x_0)\subset\Omega$,
$R\le 1$. Assume that
\begin{equation}\label{D}
\frac{\big|B_r(x_0)\cap\{u=0\}\big|}{|B_r(x_0)|}\ge c_0>0\quad\mbox{if}\quad r\le R.\end{equation}

Then, for every $x_1$ in $B_r(x_0)$,
\begin{equation}
\label{bound-grad-mismo} |\nabla u|\le \lstar(x_1)+C
\left(\frac{r}{R}\right)^\gamma\quad\mbox{in}\quad
B_r(x_1)\quad\mbox{if}\quad r\le R,
\end{equation}
for some constants $C$ and $0<\gamma<1$ depending only on $N$,
$p_{\min}$, $p_{\max}$, $\lambda_{\min}$,
$\|f\|_{L^\infty(B_{2R}(x_0))\cap W^{1,q}(B_{2R}(x_0))}$,
$\|p\|_{W^{1,\infty}(B_{2R}(x_0))\cap W^{2,q}(B_{2R}(x_0))}$,
$\alpha^*$, $C^*$, $q$, $\|\nabla u\|_{L^\infty(B_{2R}(x_0))}$ and
$c_0$.

\end{lemm}
\begin{proof}
The proof is exactly as that of Lemma \ref{Thm4.2} the only difference being that instead of the flatness condition we use the  density condition
\eqref{D}.
\end{proof}

Now, with the ideas in the proof of Lemma \ref{Thm4.2} we can improve on the gradient.
\begin{lemm} \label{lema 10} Let
$p\in W^{1,\infty}(B_\rho)\cap W^{2,q}(B_\rho)$ with
$1<p_{\min}\le p(x)\le p_{\max}<\infty$ in $B_\rho$ and $f\in
L^\infty(B_\rho)\cap W^{1,q}(B_\rho)$ with $q>\max\{1,N/2\}$,
$\|p\|_{W^{1,\infty}(B_\rho)\cap W^{2,q}(B_\rho)}\le \widetilde
L_1$ and $\|f\|_{L^\infty(B_\rho)\cap W^{1,q}(B_\rho)}\le
\widetilde L_2$. Let $\lstar\in C^{\alpha^*}(B_\rho)$ with
$0<\lambda_{\min}\le \lstar(x)\le \lambda_{\max}<\infty$ in
$B_\rho$ and $[\lstar]_{C^{\alpha^*}(B_\rho)}\le C^*$.

 Let $0<\theta<1$. There exist $\sigma_\theta$, $c_\theta$, $C$, $\tilde C$ and $\tilde\gamma$
such that, if
\[
u\in F(\sigma,1; \tau)\mbox{ in }B_\rho\mbox{ in direction }\nu
\]
with power $p$, slope $\lstar$ and rhs $f$ and, if $\sigma\le
\sigma_\theta$, $\tau\le \sigma_\theta\sigma^2$ and $\tilde
C\rho^{\tilde\gamma}\le\lambda_{\min}\tau$, there  holds that
\[
u\in F(\theta\sigma,\theta\sigma;\theta^2\tau)\mbox{ in
}B_{\bar\rho}\mbox{ in direction }\bar\nu
\]
with the same power, slope and rhs and
\[
c_\theta\rho\le\bar\rho\le\frac14\rho,\qquad |\nu-\bar\nu|\le
C\sigma.
\]
The constants depend only on  $N$,  $p_{\min}$, $p_{\max}$,
$\lambda_{\min}$, $\lambda_{\max}$, $\widetilde L_1$, $\widetilde
L_2$, $\alpha^*$, $C^*$, $q$. The constants $\sigma_\theta$ and
$c_\theta$ depend moreover on $\theta$.
\end{lemm}
\begin{proof} We  will apply  Lemma \ref{lema 8} inductively, and we will obtain the improvement of the value $\tau$ with an
argument similar to the one in
Lemma \ref{Thm4.2}.

In fact, if $\sigma_{\theta}$ is small enough, we can apply
Proposition \ref{prop-flatness} to $\bar u(x)=\frac{1}{\rho}
u(\rho x)$ and we get
$$u\in F(C_0\sigma,C_0\sigma;\tau) \mbox{ in } B_{\rho/2} \mbox{ in direction } \nu,$$
with power $p$, slope $\lstar$ and rhs $f$. Then
for $0<\theta_1\leq \frac{1}{2}$ we can apply Lemma \ref{lema 8},
if again $\sigma_{\theta}$ is small, and we obtain
\begin{equation}\label{theta1}
u\in F(C_0\theta_1\sigma,1;\tau) \mbox{ in }
B_{r_1\rho} \mbox{ in direction } \nu_1,
\end{equation}
with the same power, slope and rhs, for some $r_1,
\nu_1$ with
$$c_{\theta_1}\leq 2 r_1\leq \theta_1, \mbox{ and }
|\nu_1-\nu|\leq C\sigma.
$$

In order to improve the value of $\tau$ we proceed as in the proof
of Lemma \ref{Thm4.2}. In fact, we let $R_0=R=r_1\rho$, $x_0=0$
and repeat the argument leading to \eqref{estim-para-h(r)}, with
$r=r_1\rho$. In the present case we use the fact that, because of
\eqref{theta1}, $u$ vanishes in the ball
$B_{\frac{r_1\rho}{4}}(\frac{r_1\rho}{2}\nu_1)$. We also use that,
in $B_\rho$, $|\nabla u|\le \lstar(0)(1+\tau)\le 2\lambda_{\max}$.
We obtain
$$
\sup_{B_{r_1\rho}}\big(|\nabla u|-\lstar_{2{r_1\rho}}\big)^+ \le \big(1-\bar c\big) \sup_{B_{2r_1\rho}}\big(|\nabla u|-\lstar_{2{r_1\rho}}\big)^+
+\bar C(r_1\rho)^{2-N/q},
$$
with
$$
\lstar_{2{r_1\rho}}=\sup_{B_{2{r_1\rho}}}\lstar(x),
$$
and constants $0<\bar c<1$ and $\bar C>0$ depending only on  $N$,
$p_{\min}$, $p_{\max}$, $\lambda_{\min}$, $\lambda_{\max}$,
$\widetilde L_1$, $\widetilde L_2$  and $q$. It follows that
\[
\begin{aligned}
\sup_{B_{r_1\rho}}|\nabla u|&\le \lstar_{2{r_1\rho}}+\big(1-\bar c\big) \lstar_{2{r_1\rho}}\tau
+\bar C\big(\frac{\rho}{4}\big)^{2-N/q}\\
&\le \lstar_{2{r_1\rho}}+\big(1-\frac{\bar c}{2}\big) \lstar_{2{r_1\rho}}\tau,
\end{aligned}
\]
if we let $\bar C\big(\frac{\rho}{4}\big)^{2-N/q}\le \frac{\bar
c}{2}\lambda_{\min}\tau$. Therefore, for $\hat\theta=1-\frac{\bar
c}{2}$, we get
\[
\begin{aligned}
\sup_{B_{r_1\rho}}|\nabla u|&\le \lstar_{2{r_1\rho}}(1+\hat\theta\tau)\\
&\le \lstar(0)(1+\hat\theta\tau)+C^*(2 r_1\rho)^{\alpha^*}(1+\hat\theta\tau)\\
&\le \lstar(0)\big(1+\hat\theta\tau+\frac{1-\hat\theta}{2}\tau \big)=\lstar(0)(1+\theta_0^2\tau),
\end{aligned}
\]
if $C^*{\rho}^{\alpha^*}\le \frac{1}{2}\lambda_{\min}\tau$ and
$\theta_1^{\tilde\gamma}\le \frac{1-\hat\theta}{2}$, with
${\tilde\gamma}=\min\{\alpha^*,{2-N/q}\}$ and
$\theta_0=\sqrt{\frac{1+\hat\theta}{2}}$.

We see that, if $\theta_1$ is chosen small enough,
$$u\in F(\theta_0\sigma,1;\theta_0^{2}\tau) \mbox{ in } B_{r_1\rho} \mbox{ in
direction } \nu_1,$$
with power $p$, slope $\lstar$ and rhs $f$. Moreover, $r_1^{\tilde\gamma}\le\theta_0^{2}$.

Then, we can repeat this argument a finite number of times, and we obtain
$$u\in F(\theta_0^m\sigma,1;\theta_0^{2m}\tau) \mbox{ in } B_{r_1...r_m \rho} \mbox{ in
direction } \nu_m,$$
with the same power, slope and rhs, with
$$c_{\theta_j}\leq 2 r_j\leq \theta_j, \mbox{ and }
|\nu_m-\nu|\leq \frac{C}{1-\theta_0}\sigma.$$ Finally we choose
$m$ large enough and use Proposition \ref{prop-flatness}.
\end{proof}

\end{section}
%%%%%%%%%%%%%%%%% end section flat free boundary points%%%%%%%%%%%%%%%%%%%%%%%%%%%%%%%%%%%%%%

%%%%%%%%%%%%begin section regularity of the free boundary%%%%%%%%%%%%%%%%%%
\begin{section}{Regularity of the free boundary for weak solutions to problem $P(f,p,{\lambda}^*)$} \label{sect-regularity}

In this section we study the regularity of the free boundary for
weak solutions to problem $P(f,p,{\lambda}^*)$.

We prove that the free boundary of a weak solution   is a $C^{1,\alpha}$ surface
near flat free boundary points (Theorems \ref{reg-weak-conflat+dens}, \ref{reg-weak-conflat+flat} and \ref{reg-weak-conflat}).
As a consequence we get that  the
free boundary is  $C^{1,\alpha}$   in a neighborhood of
every point in the reduced free boundary
(Theorem \ref{reg-weak-final}).

We also obtain further regularity results on the free boundary,
under further regularity assumptions on the data (Corollary
\ref{higher-reg}).

\medskip

Among Theorems  \ref{reg-weak-conflat+dens}, \ref{reg-weak-conflat+flat} and \ref{reg-weak-conflat} the most general one is Theorem \ref{reg-weak-conflat}.

Theorems \ref{reg-weak-conflat+dens} and \ref{reg-weak-conflat+flat} require the extra assumptions \eqref{D2} and \eqref{flat-every-level}, respectively.
But, under these additional assumptions, the constant in the $C^{1,\alpha}$ continuity of the free boundary becomes universal.

The difference stems from the fact that in Theorems \ref{reg-weak-conflat+dens} and \ref{reg-weak-conflat+flat} the choice of $\rho$ in the statements can be
done independently of the weak solution $u$ under consideration, whereas in Theorem \ref{reg-weak-conflat} there is a strong dependence on $u$.

We remark that the H\"older exponent $\alpha$ is universal in the three results.

\bigskip

Our first result holds at  free boundary points satisfying a density condition on the zero set. This is the situation when
$u$ comes from a minimization problem as was the case in \cite{AC,ACF,DP1}, for instance.

\begin{theo}\label{reg-weak-conflat+dens}
Let $p\in W^{1,\infty}(\Omega)\cap W^{2,q}(\Omega)$ with
$1<p_{\min}\le p(x)\le p_{\max}<\infty$ in $\Omega$ and $f\in
L^\infty(\Omega)\cap W^{1,q}(\Omega)$
 with $q>\max\{1,N/2\}$. Let $\lstar\in C^{\alpha^*}(\Omega)$ with $0<\lambda_{\min}\le \lstar(x)\le \lambda_{\max}<\infty$ in $\Omega$ and
$[\lstar]_{C^{\alpha^*}(\Omega)}\le C^*$. Let $u$ be a weak solution to \ref{fbp-px} in $\Omega$ and let $x_0\in\Omega\cap\partial\{u>0\}$
with $B_{4R}(x_0)\subset\Omega$, $R\le 1$. Assume that
\begin{equation}\label{D2}
\frac{\big|B_r(x_0)\cap\{u=0\}\big|}{|B_r(x_0)|}\ge c_0>0\quad\mbox{if}\quad r\le R.\end{equation}

Then  there are constants $\alpha$, $\beta$, $\bar\sigma_0$,  $\bar C$ and $C$ such that if
\[
u\in F(\sigma,1; \infty)\mbox{ in }B_\rho(x_0)\mbox{ in
direction }\nu
\]
with power $p$, slope $\lstar$ and rhs $f$, with  $\sigma\le
\bar\sigma_0$ and $\bar C\rho^{\beta}\le\bar\sigma_0\sigma^2$, then
\[
B_{\rho/4}(x_0)\cap\partial\{u>0\}\mbox{ is a } C^{1,\alpha}\mbox{
surface},
\]
more precisely, a graph in direction $\nu$ of a $C^{1,\alpha}$
function, and, for $x, y$ on this surface,
\begin{equation}\label{reg-c1alpha}
|\nu( x)-\nu(y)|\leq C\sigma \left|
\frac{ x-y}{\rho}\right|^{\alpha}.
\end{equation}

The constants depend only on $N$,
 $p_{\min}$, $p_{\max}$, $\lone$, $\ltwo$, $\alpha^*$, $C^*$, $q$,
$\|f\|_{L^\infty(B_{3R}(x_0))\cap W^{1,q}(B_{3R}(x_0))}$, $\|p\|_{W^{1,\infty}(B_{3R}(x_0))\cap W^{2,q}(B_{3R}(x_0))}$, $R$,
  $c_0$ and the constants $C_{\max}(B_{3R}(x_0))$ and $r_0(B_{3R}(x_0))$ in Definition \ref{weak2}.

\end{theo}
\begin{proof}
Let us first get a bound for $\|\nabla u\|_{L^\infty(B_{2r_1}(x_0))}$ for a suitable $0<r_1\le R$. In fact, we denote $r_0=r_0(B_{3R}(x_0))$ and
$C_{\max}=C_{\max}(B_{3R}(x_0))$, the constants in Definition \ref{weak2}. We now let $r_1=\frac{1}{4}\min\{3R, r_0\}$ and see that there holds that
$\|u\|_{L^\infty(B_{4r_1}(x_0))}\le C_{\max}r_0$.

Then, by Proposition \ref{loc-lip}, it follows that $\|\nabla u\|_{L^\infty(B_{2r_1}(x_0))}$ can be estimated by a constant depending only on
$N$,  $p_{\min}$, $p_{\max}$, $r_1$,
$\|f\|_{L^\infty(B_{4r_1}(x_0))\cap W^{1,q}(B_{4r_1}(x_0))}$, $\|p\|_{W^{1,\infty}(B_{4r_1}(x_0))\cap W^{2,q}(B_{4r_1}(x_0))}$, $C_{\max}$ and $r_0$.

Next, we choose the constants in the statement so that $\rho\le r_1$. Then, we can apply Lemma \ref{stronger-Thm4.2} in $B_{4r_1}(x_0)$ and get, for $x\in B_{\rho}(x_0)$,
\begin{align*}
 |\nabla u(x)| \leq
\lambda^*(x_0)+C_1{\rho}^{\gamma}\le \lambda^*(x_0)\Big(1+\frac{C_1}{\lone}{\rho}^{\gamma}\Big),
\end{align*}
with $C_1$ and $\gamma$ constants depending only on $N$,  $p_{\min}$, $p_{\max}$, $\lone$, $\|f\|_{L^\infty(B_{2r_1}(x_0))\cap W^{1,q}(B_{2r_1}(x_0))}$,
$\|p\|_{W^{1,\infty}(B_{2r_1}(x_0))\cap W^{2,q}(B_{2r_1}(x_0))}$, $\alpha^*$, $C^*$, $q$, $\|\nabla u\|_{L^\infty(B_{2r_1}(x_0))}$, $c_0$ and
$r_1$.

We let $\bar C$ and $\beta$ in the statement satisfying $\bar C\ge \frac{C_1}{\lone}$ and $\beta\le\gamma$, and take
$\tau=\bar C{\rho}^{\beta}$. Therefore we obtain
$$u\in F(\sigma,1;\tau)  \mbox{ in } B_{\rho}(x_0) \mbox{
in direction } \nu,$$
with power $p$, slope $\lstar$ and rhs $f$.

Applying Proposition \ref{prop-flatness} we have that
\begin{equation}\label{flat-C0}
u\in
F(C_0\sigma,C_0\sigma;\tau)  \mbox{ in } B_{\rho/2}(x_0) \mbox{ in
direction } \nu,\end{equation}
with the same power, slope and rhs, if we choose $\bar C\ge C^*$, $\beta\le\alpha^*$, and
$\bar\sigma_0$ is small enough so that, in particular, $\tau\le\sigma$
and $C^*{\rho}^{\alpha^*}\le \bar C {\rho}^{\beta}\le\lambda_{\min}\sigma$.

Let $x_1\in B_{\rho/2}(x_0) \cap\partial\{u>0\}$. Since Lemma \ref{stronger-Thm4.2} also gives
\begin{align*}
 |\nabla u(x)| \leq
\lambda^*(x_1)+C_1{\rho}^{\gamma}\le \lambda^*(x_1)(1+\tau) \ \mbox{ in }  B_{\rho/2}(x_1)
\end{align*}
and   $\langle x_1-x_0,\nu\rangle>-C_0\sigma\frac\rho2$ there holds that,
$$u\in
F(\bar C_0\sigma,1;\tau)  \mbox{ in }  B_{\rho/2}(x_1) \mbox{ in direction
} \nu,$$
 with power $p$, slope $\lstar$ and rhs $f$, for any constant $\bar C_0\ge(C_0+2)$.

If we let $\bar\sigma_0$ small enough, the above choice of $\bar C$ and $\beta$, which implies in particular that
$\tau\le \bar C_0\sigma$ and $C^*{(\frac{\rho}{2})}^{\alpha^*}\le \lambda_{\min}\bar C_0\sigma$, allows us to
apply again Proposition \ref{prop-flatness} and deduce that
$$u\in
F(C\sigma,C\sigma;\tau)  \mbox{ in }  B_{\rho/4}(x_1) \mbox{ in
direction } \nu,$$
with the same power, slope and rhs.

We want to apply Lemma \ref{lema 10} in  $B_{\rho/4}(x_1)$ for some $0<\theta<1$. In fact, we need
$C\sigma\leq\sigma_{\theta}$, $\tau\leq \sigma_{\theta} (C\sigma)^2$ and $\tilde C{(\frac{\rho}{4})}^{\tilde\gamma}\le \lone\tau$,
which is satisfied if we let $\bar\sigma_0\le \frac{\sigma_{\theta}}{C}$, $\bar\sigma_0\le\sigma_{\theta} C^2$, $\bar C\ge \frac{\tilde C}{\lone}$ and
$\beta\le \tilde\gamma$.

Moreover, we want to apply Lemma \ref{lema 10}
inductively in order to get sequences  $\rho_m$ and $\nu_m$, with $\rho_0=\rho/4$ and $\nu_0=\nu$, such that
$$u\in
F(\theta^mC\sigma,\theta^m C\sigma;\theta^{2m}\tau)   \mbox{ in }
B_{\rho_m}(x_1) \mbox{ in direction } \nu_m,$$
with power $p$, slope $\lstar$ and rhs $f$,
with
\begin{equation} \label{rhom-num}
c_{\theta}
\rho_m\leq \rho_{m+1}\leq \rho_m/4\quad\mbox{ and }\quad|\nu_{m+1}-\nu_m|\leq
\theta^m C \sigma.
\end{equation}

For this purpose, we have to verify at each step that
$$
\theta^mC\sigma\leq\sigma_{\theta}, \quad \theta^{2m}\tau\leq \sigma_{\theta} (\theta^{m}C\sigma)^2, \quad
\tilde C\rho_m^{\tilde\gamma}\le \lone\theta^{2m}\tau.
$$
Since $\rho_m\le 4^{-m}\rho_0$, this is satisfied if, in addition, we let $\theta=2^{-\beta}<1$.

Thus, we have that $$|\langle x-x_1,\nu_m\rangle|\leq \theta^m C\sigma
\rho_m\quad\mbox{for}\quad  x\in B_{\rho_m}(x_1)\cap\partial\{u>0\}.$$

We also have that there exists $\nu(x_1)=\lim_{m\to\infty} \nu_m$  and
\begin{equation}\label{cota-normal}
|\nu(x_1)-\nu_m|\leq \frac{C\theta^m}{1-\theta} \sigma.
\end{equation}

Now let $x\in B_{\rho/4}(x_1)\cap\partial\{u>0\}$ and choose $m$
such that $\rho_{m+1}\leq |x-x_1|< \rho_m$. Then
$$|\langle x-x_1, \nu(x_1)\rangle|\leq C \theta^m
\sigma\Big(\frac{|x-x_1|}{1-\theta}+\rho_m\Big) \leq C \theta^m
\sigma \Big(\frac{1}{1-\theta}+\frac{1}{c_{\theta}}\Big)|x-x_1|
$$
and since $|x-x_1|\geq c_{\theta}^{m+1} \rho_0$ we have
\begin{equation}\label{elec-alpha}
\theta^{m+1}\leq \Big(\frac{|x-x_1|}{\rho_0}\Big)^{\alpha}\quad \mbox{ with } \quad
\alpha=\frac{\beta\log 2}{\log c_{\theta}^{-1}}=\frac{\log \theta}{\log c_{\theta}},
\end{equation}
and we obtain that
\begin{equation}\label{cota-fund-c1alpha}
|\langle x-x_1, \nu(x_1)\rangle|\leq
\frac{C\sigma}{\rho^{\alpha}}|x-x_1|^{1+\alpha},\qquad x\in B_{\rho/4}(x_1) \cap\partial\{u>0\}.
\end{equation}

\medskip

 Let us finally observe that the result in the statement follows if we take
$\bar{\sigma}_0$ small enough.

In fact, \eqref{cota-fund-c1alpha} implies that $\nu(x_1)$ is the normal to $\partial\{u>0\}$ at $x_1$.

From  \eqref{flat-C0}, \eqref{cota-fund-c1alpha} and  \eqref{cota-normal} with $m=0$ we get that $B_{\rho/4}(x_0)\cap\partial\{u>0\}$ is  a graph
in the direction $\nu$ of a function $g$ that is defined, differentiable and Lipschitz   in $B'_{\rho/4}(x_0')$. This holds if $\bar \sigma_0$ is small so that
$$\sqrt{1-(C_0\sigma)^2}\ge 1/2\quad\mbox{and}\quad C\sigma\Big(1+\frac1{1-\theta}\Big)\le1/2
\qquad\mbox{for}\quad\sigma\le \bar \sigma_0.$$

With these choices, the Lipschitz constant of $g$ is universal (observe that \eqref{flat-C0} implies that $|g(x')-g(x_1')|\le C_0\sigma\rho$ if
$x',x'_1\in B'_{\rho/4}(x'_0) $).

\medskip

In order to see that \eqref{reg-c1alpha} holds we let $x,y\in B_{\rho/2}(x_0) \cap\partial\{u>0\}$ such that $|x-y|<\rho/8$.

We can apply the construction above with $x_1=y$, so we have sequences $\rho_m=\rho_m(y)$ with $\rho_0(y)=\rho/4$, and $\nu_m=\nu_m(y)$ satisfying \eqref{rhom-num}, with
$\nu(y)=\lim_{m\to\infty} \nu_m(y)$.

Now let $m_0$ be such that
\begin{equation}\label{elec-m0}
\frac{\rho_{m_0+1}}2\le |x-y|<\frac{\rho_{m_0}}2.
\end{equation}
We use that
\begin{equation}\label{flat-en-y}
u\in
F({\sigma}_{m_0},{\sigma}_{m_0};{\tau}_{m_0})   \mbox{ in }
B_{\rho_{m_0}}(y) \mbox{ in direction } \nu_{m_0}(y),
\end{equation}
with power $p$, slope $\lstar$ and rhs $f$, for ${\sigma}_{m_0}=\theta^{m_0}C\sigma$ and ${\tau}_{m_0}=\theta^{2{m_0}}\tau$.

\medskip

In fact, we have now the following picture: $u$ is under the assumption of the theorem with $x_0$   replaced by $y$ and flatness condition \eqref{flat-en-y}. Then, with $x_1$ replaced by $x$, $\rho_0(x)=\rho_{m_0}(y)$ and
$\nu_0(x)=\nu_{m_0}(y)$, \eqref{cota-normal} with $m=0$ gives
\begin{equation*}
|\nu(x)-\nu_{m_0}(y)|=|\nu(x)-\nu_{0}(x)|\leq \frac{C{\sigma}_{m_0}}{1-\theta}.
\end{equation*}
Let us notice  that, from the choice of $\alpha$ we made in
\eqref{elec-alpha}, ${\sigma}_{m_0}=C\sigma\theta^{m_0}=C\sigma (c_{\theta}^{m_0})^{\alpha}$.
Since, by \eqref{rhom-num} and \eqref{elec-m0}, $c_{\theta}^{m_0+1}\le 4\frac{\rho_{m_0+1}}{\rho}\le \frac{8}{\rho}|x-y|$,
there holds
\begin{equation*}
|\nu(x)-\nu_{m_0}(y)|\leq \frac{C{\sigma}}{1-\theta}\left(\frac{8|x-y|}{c_{\theta}\rho }\right)^{\alpha}.
\end{equation*}
Estimate \eqref{cota-normal} also gives
\begin{equation*}
|\nu(y)-\nu_{m_0}(y)|\leq \frac{C{\sigma}}{1-\theta}\left(\frac{8|x-y|}{c_{\theta}\rho }\right)^{\alpha}.
\end{equation*}
We thus get
\begin{equation*}
|\nu(x)-\nu(y)|\leq C\sigma \left|
\frac{x-y}{\rho}\right|^{\alpha}\quad\mbox{if}\quad x,y\in B_{\rho/2}(x_0) \cap\partial\{u>0\},\quad |x-y|<\rho/8.
\end{equation*}

Finally, if $x,y\in B_{\rho/4}(x_0)\cap\partial\{u>0\}$ are such that $|x-y|\ge\rho/8$ we can find  points
$z_i\in B_{\rho/4}(x_0)\cap\partial\{u>0\}$ with $z_0=x$, $z_k=y$, $ |z_i-z_{i+1}|<\rho/8$ for every $i$ and $k$ a universal number. By applying the last estimate we get \eqref{reg-c1alpha}.

So, the theorem is proved.
\end{proof}

In the next result  we replace the density condition \eqref{D2} of Theorem \ref{reg-weak-conflat+dens} by a flatness condition at
the point, at every scale. In fact, we get

\begin{theo}\label{reg-weak-conflat+flat}
Let $p\in W^{1,\infty}(\Omega)\cap W^{2,q}(\Omega)$ with
$1<p_{\min}\le p(x)\le p_{\max}<\infty$ in $\Omega$ and $f\in
L^\infty(\Omega)\cap W^{1,q}(\Omega)$
 with $q>\max\{1,N/2\}$. Let $\lstar\in C^{\alpha^*}(\Omega)$ with $0<\lambda_{\min}\le \lstar(x)\le \lambda_{\max}<\infty$ in $\Omega$ and
$[\lstar]_{C^{\alpha^*}(\Omega)}\le C^*$. Let $u$ be a weak solution to \ref{fbp-px} in $\Omega$ and let $x_0\in\Omega\cap\partial\{u>0\}$
with $B_{4R}(x_0)\subset\Omega$, $R\le 1$.  Assume that, for every $r\le R$,
\begin{equation}\label{flat-every-level}
u\in F(1/2,1;\infty)\quad\mbox{in}\quad B_r(x_0)\quad\mbox{in some direction}\quad\nu_r,
\end{equation}
with power $p$, slope $\lstar$ and rhs $f$.

Then  there are constants $\alpha$, $\beta$, $\bar\sigma_0$,  $\bar C$ and $C$ such that if
\[
u\in F(\sigma,1; \infty)\mbox{ in }B_\rho(x_0)\mbox{ in
direction }\nu
\]
with power $p$, slope $\lstar$ and rhs $f$, with  $\sigma\le
\bar\sigma_0$ and $\bar C\rho^{\beta}\le\bar\sigma_0\sigma^2$, then
\[
B_{\rho/4}(x_0)\cap\partial\{u>0\}\mbox{ is a } C^{1,\alpha}\mbox{
surface},
\]
more precisely, a graph in direction $\nu$ of a $C^{1,\alpha}$
function, and, for $x, y$ on this surface,
$$|\nu( x)-\nu(y)|\leq C\sigma \left|
\frac{x-y}{\rho}\right|^{\alpha}.$$

The constants depend only on $N$,
 $p_{\min}$, $p_{\max}$, $\lone$, $\ltwo$, $\alpha^*$, $C^*$, $q$,
$\|f\|_{L^\infty(B_{3R}(x_0))\cap W^{1,q}(B_{3R}(x_0))}$, $\|p\|_{W^{1,\infty}(B_{3R}(x_0))\cap W^{2,q}(B_{3R}(x_0))}$, $R$
   and the constants $C_{\max}(B_{3R}(x_0))$ and $r_0(B_{3R}(x_0))$ in Definition \ref{weak2}.
\end{theo}

\begin{proof}
The proof is exactly as that of Theorem  \ref{reg-weak-conflat+dens} the only difference being that instead of
using Lemma \ref{stronger-Thm4.2}, we make use of Lemma \ref{Thm4.2}.
\end{proof}

Our last result on the regularity
of the free boundary of a weak solution in a neighborhood of a
flat free boundary point holds without the extra assumptions \eqref{D2} and \eqref{flat-every-level} of Theorems \ref{reg-weak-conflat+dens} and
\ref{reg-weak-conflat+flat}. In fact, we get

\begin{theo}\label{reg-weak-conflat} Let $p\in W^{1,\infty}(\Omega)\cap W^{2,q}(\Omega)$ with
$1<p_{\min}\le p(x)\le p_{\max}<\infty$ in $\Omega$ and $f\in
L^\infty(\Omega)\cap W^{1,q}(\Omega)$
 with $q>\max\{1,N/2\}$. Let $\lstar\in C^{\alpha^*}(\Omega)$ with $0<\lambda_{\min}\le \lstar(x)\le \lambda_{\max}<\infty$ in $\Omega$ and
$[\lstar]_{C^{\alpha^*}(\Omega)}\le C^*$. Let $u$ be a weak solution to \ref{fbp-px} in $\Omega$ and let $x_0\in\Omega\cap\partial\{u>0\}$.

Then  there are constants $\alpha$, $\bar\sigma_0$ and $C$ such that if
\[
u\in F(\sigma,1; \infty)\mbox{ in }B_\rho(x_0)\mbox{ in
direction }\nu
\]
with power $p$, slope $\lstar$ and rhs $f$, with  $\sigma\le
\bar\sigma_0$ and $\rho$ small enough, then
\[
B_{\rho/4}(x_0)\cap\partial\{u>0\}\mbox{ is a } C^{1,\alpha}\mbox{
surface},
\]
more precisely, a graph in direction $\nu$ of a $C^{1,\alpha}$
function, and, for $x, y$ on this surface,
\begin{equation*}
|\nu( x)-\nu(y)|\leq C\sigma \left|
\frac{ x-y}{\rho}\right|^{\alpha}.
\end{equation*}

The constants $\alpha$, $\bar\sigma_0$ and $C$ depend only on $N$,
 $p_{\min}$, $p_{\max}$,
$\|f\|_{L^\infty(\Omega)\cap W^{1,q}(\Omega)}$, $\|p\|_{W^{1,\infty}(\Omega)\cap W^{2,q}(\Omega)}$,  $\lone$, $\ltwo$, $\alpha^*$, $C^*$ and $q$.
\end{theo}
\begin{proof}
Since
\begin{align*}
& \limsup_{\stackrel{x\to x_0}{u(x)>0}} |\nabla u(x)| \leq
\lambda^*(x_0),
\end{align*}
given $\bar\sigma_0$ and $\sigma\le\bar\sigma_0$, there exists $\rho_1=\rho_1(u, x_0, \bar\sigma_0,\sigma,\lambda_{\min})$ such that, if $\rho\le\rho_1$,
\begin{align}\label{cota-en-x0}
 |\nabla u(x)| \leq
 \lambda^*(x_0)\Big(1+\frac{\bar\sigma_0 \sigma^2}{2}\Big), \quad \mbox{ for } x\in B_{\rho}(x_0).
\end{align}

We take $\tau=\bar\sigma_0 \sigma^2$ and  obtain
$$u\in F(\sigma,1;\tau)  \mbox{ in } B_{\rho}(x_0) \mbox{
in direction } \nu,$$
with power $p$, slope $\lstar$ and rhs $f$.

Applying Proposition \ref{prop-flatness} we have that
$$u\in
F(C_0\sigma,C_0\sigma;\tau)  \mbox{ in } B_{\rho/2}(x_0) \mbox{ in
direction } \nu,$$
with the same power, slope and rhs, if
$\bar\sigma_0$ is small enough so that, in particular, $\tau\le\sigma$
and $\rho\le\rho_2(C^*,\alpha^*, \lambda_{\min}, \sigma)$ so that  $C^*{\rho}^{\alpha^*}\le\lambda_{\min}\sigma$.

Let $x_1\in B_{\rho/2}(x_0) \cap\partial\{u>0\}$. From \eqref{cota-en-x0} and the H\"older continuity of $\lambda^*(x)$ we get
\begin{align*}
 |\nabla u(x)| \leq
\big(\lambda^*(x_1)+C^*{({\rho}/{2})}^{\alpha^*}\big)\Big(1+\frac{\bar\sigma_0 \sigma^2}{2}\Big)\le \lambda^*(x_1)(1+\tau) \quad\mbox{ in }  B_{\rho/2}(x_1),
\end{align*}
if $\rho\le\rho_3(C^*,\alpha^*, \lambda_{\min},\bar\sigma_0, \sigma)$, so that $C^*{({\rho}/{2})}^{\alpha^*}\le\lambda_{\min}\frac{\bar\sigma_0 \sigma^2}{4}$.

Then,
$$u\in
F(\bar C_0\sigma,1;\tau)  \mbox{ in }  B_{\rho/2}(x_1) \mbox{ in direction
} \nu,$$
 with power $p$, slope $\lstar$ and rhs $f$, for any constant $\bar C_0\ge C_0+2$.

If we let $\bar\sigma_0$ small enough, so that, in particular, $\tau\le \bar C_0\sigma$,
and take $\rho\le\rho_4(C^*,\alpha^*, \lambda_{\min},\bar C_0, \sigma)$ so that $C^*{(\frac{\rho}{2})}^{\alpha^*}\le \lambda_{\min}\bar C_0\sigma$,
we can apply again Proposition \ref{prop-flatness} and deduce that
$$u\in
F(C\sigma,C\sigma;\tau)  \mbox{ in }  B_{\rho/4}(x_1) \mbox{ in
direction } \nu,$$
with the same power, slope and rhs.

We want to apply Lemma \ref{lema 10} in  $B_{\rho/4}(x_1)$ for some $0<\theta<1$. In fact, we need
$C\sigma\leq\sigma_{\theta}$, $\tau\leq \sigma_{\theta} (C\sigma)^2$ and $\tilde C{(\frac{\rho}{4})}^{\tilde\gamma}\le \lone\tau$,
which is satisfied if we let $\bar\sigma_0\le \frac{\sigma_{\theta}}{C}$, $\bar\sigma_0\le\sigma_{\theta} C^2$ and
$\rho\le\rho_5(\tilde C,\tilde\gamma, \lambda_{\min},\bar\sigma_0, \sigma).$

Moreover, we want to apply Lemma \ref{lema 10}
inductively in order to get sequences  $\rho_m$ and $\nu_m$, with $\rho_0=\rho/4$ and $\nu_0=\nu$, such that
$$u\in
F(\theta^mC\sigma,\theta^m C\sigma;\theta^{2m}\tau)   \mbox{ in }
B_{\rho_m}(x_1) \mbox{ in direction } \nu_m,$$
with power $p$, slope $\lstar$ and rhs $f$,
with $c_{\theta}
\rho_m\leq \rho_{m+1}\leq \rho_m/4$ and $|\nu_{m+1}-\nu_m|\leq
\theta^m C \sigma$.

For this purpose, we have to verify at each step
$$
\theta^mC\sigma\leq\sigma_{\theta}, \quad \theta^{2m}\tau\leq \sigma_{\theta} (\theta^{m}C\sigma)^2, \quad
\tilde C\rho_m^{\tilde\gamma}\le \lone\theta^{2m}\tau.
$$
Since $\rho_m\le 4^{-m}\rho_0$, this is satisfied if, in addition, we let $\theta=2^{-\tilde\gamma}<1$.

Now the proof follows as that of Theorem \ref{reg-weak-conflat+dens}, with $\alpha=\frac{\tilde\gamma\log 2}{\log c_{\theta}^{-1}}$, and the conclusion is
obtained if $\rho\le\bar\rho_0=\min\{\rho_1, \rho_2, \rho_3, \rho_4, \rho_5\}$.
\end{proof}

\medskip

As a consequence  of Theorem  \ref{reg-weak-conflat} we obtain

\begin{theo} \label{reg-weak-final} Let $f$, $p$ and $\lambda^*$ be as in Theorem \ref{reg-weak-conflat}.
Let $u$ be a weak solution of \ref{bernoulli-px} in $\Omega$ and let
$x_0\in\Omega\cap\partial_{\rm{red}}\{u>0\}$. There exists $\bar r_0>0$
such that $B_{\bar r_0}(x_0)\cap\partial\{u>0\}$ is a $C^{1,\alpha}$
surface for some $0<\alpha<1$. It follows that, for some $0<\gamma<1$, 
$u$ is $C^{1,\gamma}$ up to $B_{\bar r_0}(x_0)\cap\partial\{u>0\}$ and the free boundary
condition is satisfied in the classical sense. In addition, for every $x_1\in B_{\bar r_0}(x_0)\cap\partial\{u>0\}$ there
is a neighborhood $\mathcal U$ such that $\nabla u\neq0$
in  $\mathcal U\cap\{u>0\}$, $u\in W_{\rm{loc}}^{2,2}(\mathcal
U\cap\{u>0\})$ and the equation is satisfied in a pointwise sense
in $\mathcal U\cap\{u>0\}$.

If moreover $\nabla p$ and $f$ are
H\"older continuous in $\Omega$, then $u\in C^2(\mathcal
U\cap\{u>0\})$ and the equation is satisfied in the classical
sense in $\mathcal U\cap\{u>0\}$.
\end{theo}
\begin{proof}
The result follows from Theorem \ref{reg-weak-conflat}, by applying Lemma \ref{lema-reduced-Thm4.2} at the point $x_0$.

The $C^{1,\gamma}$ smoothness of $u$ up to $\partial\{u>0\}$, for some $0<\gamma<1$, follows from the regularity results up to the boundary of \cite{Fan} (see Theorem 1.2 in \cite{Fan}).
\end{proof}

We can also obtain higher regularity of $\partial\{u>0\}$ if the data are smoother. We have
\begin{coro}\label{higher-reg}
Let $u$, $x_0$ and $\bar r_0$ be as in Theorem \ref{reg-weak-final}.
Assume moreover that $p\in C^2(\Omega)$, $f\in C^1(\Omega)$ and
$\lambda^*\in C^2(\Omega)$, then
$B_{\bar r_0}(x_0)\cap\partial\{u>0\}\in C^{2,\mu}$ for every
$0<\mu<1$. If $p\in C^{m+1,\mu}(\Omega)$, $f\in C^{m,\mu}(\Omega)$
and $\lambda^*\in C^{m+1,\mu}(\Omega)$ for some $0<\mu<1$ and
$m\ge1$, then $B_{\bar r_0}(x_0)\cap\partial\{u>0\}\in C^{m+2,\mu}$.

Finally, if $p$, $f$ and $\lambda^*$ are analytic, then
$B_{\bar r_0}(x_0)\cap\partial\{u>0\}$ is analytic.
\end{coro}
\begin{proof} As in Theorem 8.4 in \cite{AC}, Theorem 6.3 and Remark 6.4 in \cite{ACF} and Corollary 9.2 in
\cite{DP1}, we use Theorem 2 in \cite{KN}.

In fact, we apply this theorem with our equation seen in the form $F(x,u,Du, D^2u)=0$, with
\[F(x,s,q,M)=
|q|^{p(x)-2}\Big[\sum_{ij}(\delta_{ij}+(p(x)-2)\frac{q{_i}q_{_j}}{|q|^2})M_{ij}+\sum_j p_{x_j}(x)\log|q| q_{_j}\Big]-f(x),
\]
in a neighborhood of the free boundary where $|\nabla u|\ge \frac{\lone}{2}$, and boundary condition in the form $g(x,Du)=0$,
with 
\[g(x,q)=|q|^2-{\lambda^*}^2(x).
\]
Already in \cite{AC} it was observed that Theorem 2 in \cite{KN} holds with $u\in C^2$ in $\{u>0\}$ and $u\in C^{1,\gamma}$ up to
$\partial\{u>0\}$, even though the result in \cite{KN} is stated with $u\in C^2$ up to $\partial\{u>0\}$.
\end{proof}

\end{section}
%%%%%%%%%%%end section regularity of the free boundary%%%%%%%%%%%%%%%%%%%%%%%%%%%%%%%%%%
%%%%%%%%%%%begin section application to a singular perturbation problem%%%%%%%%%%%%%%%%%%%%%%%%%%%
\begin{section}{Application to a singular perturbation problem}

In this section we apply the regularity results obtained in the
previous section to a singular perturbation problem we studied in
\cite{LW4}. Our regularity results apply to limit functions satisfying suitable conditions that are fulfilled, for instance, under the situation we
considered in \cite{LW5}.

For a different application of these regularity results we refer to our work \cite{LW5}.

We next consider the following singular pertubation
 problem for the $p_\ep(x)$-Laplacian:
\begin{equation}
\label{eq}\tag{$P_\ep(\fep, p_\ep)$}
\Delta_{p_\ep(x)}\uep={\beta}_{\varepsilon}(\uep)+\fep, \quad
u^{\ep}\geq 0
\end{equation}
in a domain $\Omega\subset \Bbb R^{N}$. Here $\ep>0$,
${\beta}_{\varepsilon}(s)={1 \over \varepsilon} \beta({s \over
\varepsilon})$, with $\beta$  a  Lipschitz function satisfying
$\beta>0$ in $(0,1)$, $\beta\equiv 0$ outside $(0,1)$ and $\int
\beta(s)\, ds=M$.

We assume that $1<p_{\min}\le p_\ep(x)\le p_{\max}<\infty$,
$\|\nabla p_\ep\|_{L^{\infty}}\leq L$ and that  the functions
$\uep$ and $\fep$ are uniformly bounded.

In \cite{LW4} we proved local uniform Lipschitz regularity for
solutions of this problem, we passed to the limit $(\ep\to 0)$ and
we showed that, under suitable assumptions, limit functions are
weak solutions to the free boundary problem: $u\ge 0$ and
\begin{equation}
\label{bernoulli-px}\tag{$P(f,p,{\lambda}^*)$}
\begin{cases}
\Delta_{p(x)}u= f & \mbox{in }\{u>0\}\\
u=0,\ |\nabla u| = \lambda^*(x) & \mbox{on }\partial\{u>0\}
\end{cases}
\end{equation}
with $\lambda^*(x)=\Big(\frac{p(x)}{p(x)-1}\,M\Big)^{1/p(x)}$, $p=\lim p_\ep$ and
$f=\lim \fep$.

\bigskip

Before giving the precise statement of one of the results we
proved in \cite{LW4}, we need the following definitions

\begin{defi} Let $u$ be a continuous nonnegative function in a domain $\Omega\subset
\Bbb R^{N}$. Let $x_0\in\Omega\cap\partial\{u>0\}$. We say that
$x_0$ is a regular point from the positive side if there is a ball
$B\subset\{u>0\}$ with $x_0\in\partial B$.
\end{defi}

\begin{defi} Let $u$ be a continuous nonnegative function in a domain $\Omega\subset
\Bbb R^{N}$. Let $x_0\in\Omega\cap\partial\{u>0\}$.

We say that condition (D) holds at $x_0$ if there exist $\gamma>0$
and $0<c<1$ such that, for every $x\in
B_\gamma(x_0)\cap\partial\{u>0\}$ which is regular from the
positive side and $r\le\gamma$, there holds that $|\{u=0\}\cap
B_r(x)|\ge c|B_r(x)|$.
\end{defi}

\begin{defi} Let $u$ be a continuous nonnegative function in a domain $\Omega\subset
\Bbb R^{N}$. Let $x_0\in\Omega\cap\partial\{u>0\}$.

We say that condition (L) holds at $x_0$ if there exist
$\gamma>0$, $\theta>0$ and $s_0>0$ such that for every point $y\in
B_\gamma(x_0)\cap\partial\{u>0\}$ which is regular from the
positive side, and for every ball $B_r(z)\subset\{u>0\}$ with
$y\in\partial B_r(z)$ and $r\le\gamma$, there exists a  unit
vector ${\tilde e}_y$,  with $\langle{\tilde
e}_y,z-y\rangle>\theta||z-y||$, such that $u(y-s{\tilde e}_y)=0$
for $0< s< s_0.$
\end{defi}

{}In \cite{LW4} we obtained the following result:

\begin{theo}\label{lim=weak}
 Let $u^{\ep_j}$ be a family of solutions to
 $\pepj(\fepj,
p_{\epj})$ in a domain $\Omega\subset \Bbb R^{N}$ with
$1<p_{\min}\le p_{\epj}(x)\le p_{\max}<\infty$ and $p_{\epj}(x)$
Lipschitz continuous with $\|\nabla p_{\epj}\|_{L^{\infty}}\leq
L$, for some $L>0$. Assume that $u^{\ep_j}\rightarrow u$ uniformly on compact subsets of $\Omega$,
$\fepj\rightharpoonup f$ $*-$weakly in $L^\infty(\Omega)$, $p_{\epj}\to p$ uniformly on compact subsets
of $\Omega$ and $\ep_j\to 0$.

Assume that $u$ is locally uniformly nondegenerate on
$\Omega\cap\partial\{u>0\}$ and that at every point
$x_0\in\Omega\cap\partial\{u>0\}$ either condition (D) or
condition (L)  holds.

Then,  $u$ is a weak solution to the free boundary problem:
$u\ge0$ and
\begin{equation}
\tag{$P(f,p,{\lambda}^*)$}
\begin{cases}
\Delta_{p(x)}u = f & \mbox{in }\{u>0\}\\
u=0,\ |\nabla u| = \lambda^*(x) & \mbox{on }\partial\{u>0\}
\end{cases}
\end{equation}
with $\lambda^*(x)=\Big(\frac{p(x)}{p(x)-1}\,M\Big)^{1/p(x)}$ and
$M=\int \beta(s)\, ds$.
\end{theo}

\medskip

\begin{rema}
In \cite{LW5} we proved that if $\uepj$, $\fepj$, $p_{\epj}$,
$\ep_j$,  $f$ and $p$ are as  in Theorem \ref{lim=weak} and
$u^{\ep_j}\rightarrow u$ uniformly on compact subsets of $\Omega$
with $\uepj$ local minimizers of an energy functional, then  $u$
is under the assumptions of Theorem~\ref{lim=weak}.
\end{rema}

\medskip

As a first application of Theorem \ref{reg-weak-final} we obtain
the following result on the regularity of the free boundary for
limit functions of the singular perturbation problem
$\pepj(\fepj,p_{\epj})$.

\begin{theo}\label{reg-lim-singpert}
Let $\uepj$, $\fepj$, $p_{\epj}$, $\ep_j$, $u$, $f$ and $p$ be as
in Theorem \ref{lim=weak}. Assume moreover that $f\in W^{1,q}(\Omega)$ and $p\in W^{2,q}(\Omega)$
 with $q>\max\{1,N/2\}$.

Let $x_0\in\Omega\cap\partial_{\rm{red}}\{u>0\}$. Then, there
exists $\bar r_0>0$ such that $B_{\bar r_0}(x_0)\cap\partial\{u>0\}$ is a
$C^{1,\alpha}$ surface for some $0<\alpha<1$. It follows that, for some $0<\gamma<1$, 
$u$ is $C^{1,\gamma}$ up to $B_{\bar r_0}(x_0)\cap\partial\{u>0\}$ and the free
boundary condition is satisfied in the classical sense. In
addition, for every $x_1\in B_{\bar r_0}(x_0)\cap\partial\{u>0\}$ there is a neighborhood $\mathcal U$ such that
$\nabla u\neq0$ in  $\mathcal U\cap\{u>0\}$, $u\in
W_{\rm{loc}}^{2,2}(\mathcal U\cap\{u>0\})$ and the equation is
satisfied in a pointwise sense in $\mathcal U\cap\{u>0\}$.

If
moreover $\nabla p$ and $f$ are  H\"older
continuous in $\Omega$, then $u\in C^2(\mathcal U\cap\{u>0\})$ and
the equation is satisfied in the classical sense in $\mathcal
U\cap\{u>0\}$.
\end{theo}

\begin{proof} The result follows from the application of Theorems  \ref{lim=weak} and \ref{reg-weak-final} above.
\end{proof}
We also obtain higher regularity from the application of  Corollary \ref{higher-reg}.
\begin{coro}\label{higher-reg-lim}
Let $u$, $x_0$ and $\bar r_0$ be as in Theorem \ref{reg-lim-singpert}. Assume moreover that $p\in C^2(\Omega)$ and $f\in C^1(\Omega)$, then
$B_{\bar r_0}(x_0)\cap\partial\{u>0\}\in C^{2,\mu}$ for every
$0<\mu<1$.
If $p\in C^{m+1,\mu}(\Omega)$ and $f\in
C^{m,\mu}(\Omega)$  for some $0<\mu<1$ and $m\ge1$, then
$B_{\bar r_0}(x_0)\cap\partial\{u>0\}\in C^{m+2,\mu}$.

Finally, if $p$ and $f$ are analytic, then $B_{\bar r_0}(x_0)\cap\partial\{u>0\}$ is analytic.
\end{coro}

\end{section}

\bigskip
%%%%%%%%%%%end section application to a singular perturbation problem%%%%%%%%%%%%%%%%%%%%%%%%%%%%%%%%%%

%%%%%%%%%%%%%%begin references%%%%%%%%%%%%%%%%%%%%%%%%%%%%%%%%%

%%%%%%%%% end of references %%%%%%%%%%%%%%%%%%%%%%%%%%%%%%%

\end{document}